\documentclass[reqno,11pt]{amsart}
\usepackage[colorlinks=true, linkcolor=blue, citecolor=blue]{hyperref}

\usepackage{amssymb}   
\usepackage{amsmath, graphicx, rotating}
\usepackage{color}     
\usepackage{soul}
\usepackage[dvipsnames]{xcolor}     
\usepackage{tcolorbox}
\usepackage{comment}

\usepackage{ifthen}
\usepackage{xkeyval}
\usepackage{todonotes}
\usepackage{lmodern}
\usepackage[english]{babel}

\usepackage{ upgreek }
\usepackage{stmaryrd}
\SetSymbolFont{stmry}{bold}{U}{stmry}{m}{n}
\usepackage{amsthm}
\usepackage{float}

\usepackage{ bbm }
\usepackage{ stmaryrd }
\usepackage{ mathrsfs }
\usepackage{ frcursive }
\usepackage{ comment }

\usepackage{pgf, tikz}
\usetikzlibrary{shapes}
\usepackage{varioref}
\usepackage{enumitem}
\usepackage{longtable}

\usepackage{mathtools}

\usepackage{dsfont}

\definecolor{rouge}{rgb}{0.7,0.00,0.00}
\definecolor{vert}{rgb}{0.00,0.5,0.00}
\definecolor{bleu}{rgb}{0.00,0.00,0.8}

\usepackage[margin=1.3in]{geometry}

\newtheorem{theorem}{Theorem}[section]
\newtheorem*{theorem*}{Theorem}
\newtheorem{lemma}[theorem]{Lemma}

\newtheorem{corollary}[theorem]{Corollary}

\labelformat{hypothesis}{\textbf{M\kern-0.1mm#1}}

\newtheorem{condition}{Condition}

\labelformat{conditionA}{\textbf{A\kern-0.1mm#1}}

\renewcommand\dots{\hbox to 1em{.\hss.\hss.}}

\theoremstyle{definition}

\def \eref#1{\hbox{(\ref{#1})}}

\numberwithin{equation}{section}

\newcommand\ee{\varepsilon}

\DeclareMathOperator{\supp}{supp}

\def\geq{\geqslant}
\def\leq{\leqslant}

\def\bb#1{\mathbb{#1}}

\def\Rd {\bb R ^d}
\def\Rd*{(\bb R ^d)^*}
\def\Pd{{\mathbb{P}}^{d-1}}
\def\Pd*{(\mathbb{P}^{d-1})^*}
\def\PV{{\mathbb{P}}({\mathbb{V}}) }
\def\PV*{({\mathbb{P}}({\mathbb{V}}))^* }

\begin{document}
\title[Conditioned random walks]
{Gaussian heat kernel asymptotics for\\conditioned random walks} 
\author{ Ion~Grama}
\curraddr[I. Grama]{ Universit\'{e} de Bretagne Sud, CRYC, 56017 Vannes, France}
\email[I. Grama]{ion.grama@univ-ubs.fr}
\author{Hui Xiao}
\curraddr[H. Xiao]{Academy of Mathematics and Systems Science, Chinese Academy of Sciences, Beijing 100190, China}
\email[H. Xiao]{xiaohui@amss.ac.cn}

\date{\today }

\begin{abstract}   
Consider a random walk $S_n=\sum_{i=1}^n X_i$ with independent and identically distributed real-valued increments 
with  zero mean, finite variance and moment of order $2 + \delta$ for some $\delta>0$. 
For any starting point $x\in \mathbb R$, let $\tau_x  = \inf \left\{ k\geq 1: x+S_{k} < 0 \right\}$ 
denote the first time when the random walk $x+S_n$ 
exits the half-line $[0,\infty)$. We investigate the uniform asymptotic behavior over $x\in \mathbb R$
 of the persistence probability $\mathbb P (\tau_x >n)$ and the joint distribution 
$\mathbb{P} \left( x + S_n  \leq u,  \tau_x > n \right)$, for $u\geq 0$,  as $n \to \infty$.
New limit theorems for these probabilities are established based on the heat kernel approximations.  
Additionally, we evaluate the rate of convergence by proving Berry-Esseen type bounds. 
\end{abstract}

\maketitle


\section{Introduction and main results}\label{Ssec-Intro}
\subsection{Preliminaries and notation}\label{Subsec-Intro}
Assume that on the probability space $\left( \Omega ,\mathscr{F},\mathbb{P}\right)$ 
we are given a sequence of independent identically distributed real-valued random variables $(X_i)_{i\geq 1}$ 
with $\bb E (X_1) = 0$ and  $\bb E (X_1^2) = \sigma^2 \in (0, \infty)$.   
Define the random walk $S_n$ by
\begin{equation} \label{Sum-iid-000}
S_n =\sum_{i=1}^{n}X_{i}, \quad n \geq 1.  
\end{equation}
For any starting point $x\in \bb R:=(-\infty,\infty)$, 
consider the first moment $\tau_x$ when the random walk $(x+S_n )_{n\geq 1}$ 
exits the non-negative half-line $\bb R_+ := [0,\infty)$, 
which is defined as  
\begin{align*}
\tau_x = \inf \left\{ k \geq 1: x+S_{k} < 0 \right\}
 \quad \mbox{with} \quad   \inf\emptyset = \infty.   
\end{align*}
For $n\geq 1$ and $x \in \bb R$ possibly depending on $n$ and $u \in \bb R$, 
consider the probabilities 
\begin{align} \label{Objective-proba001}
\mathbb{P} \left( \tau_x >n\right) \quad \mbox{and} 
\quad 
\mathbb{P} \left(  \frac{ x + S_n }{\sigma\sqrt{n}} \leq  u,  \tau_x >n\right). 
\end{align} 
The behavior of the random walks conditioned to stay positive and, specifically, 
the asymptotics of the probabilities \eqref{Objective-proba001} has been an area of interest for many researchers. 
Notable contributions include works by 
L\'{e}vy \cite{Levy37},    
Borovkov \cite{Borovk62, Bor70, Borovkov04a, Borovkov04b},   
Feller \cite{Fel64}, 
Spitzer \cite{Spitzer},
Kozlov \cite{Kozlov76}, 
Bolthausen \cite{Bolth}, 
Iglehart \cite{Igle74}, 
Eppel \cite{Eppel-1979},  
Bertoin and Doney \cite{BertDoney94},  
Vatutin and Wachtel \cite{VatWacht09}, 
Doney and Jones \cite{DJ12},
Denisov and Wachtel \cite{Den Wachtel 2011, DW19, DW24},
Kersting and Vatutin \cite{KV17}, 
 Denisov, Sakhanenko and Wachtel \cite{DSW18}, 
Grama, Lauvergnat and Le Page \cite{GLL18Ann}, 
and the references cited in these works. 
In most of these studies, the asymptotic properties of the probabilities in \eqref{Objective-proba001} 
were analyzed for the case where the starting point $x\geq 0$ is fixed.  
However, far fewer results address situations where this condition does not hold. 
For $x$ depending on $n$, two primary scenarios have been examined. 
The first scenario concerns the case when the starting point $x$ 
is such that $\frac{x}{\sqrt{n}}\to 0$ as $n\to\infty$. 
In the second scenario, the starting point $x$ is of the order $\sqrt{n}$, i.e., $x \asymp \sqrt{n}$. 
Both scenarios have been considered in the author's paper \cite{GX-2024-AIHP}, 
where, for each scenario, conditioned central limit theorems for both probabilities in \eqref{Objective-proba001} have been 
stated. We also refer to  Doney \cite{Don12} and to \cite{GX-2024-AIHP} for results in the context of conditioned local limit theorems.  
It is also worth mentioning that Nagaev \cite{Nag69, Nag70} and Aleskevicene \cite{Aleskev73} proved Berry-Esseen type bounds for the probability $\mathbb{P} \left( \tau_x >n\right)$. 
However, these Berry-Esseen bounds become efficient essentially for $x=x_n\to\infty$ as $n\to\infty$, 
thereby covering basically the case when $x \asymp \sqrt{n}$ asymptotically. 

 In this paper, we aim to study the asymptotic behavior of the probabilities in \eqref{Objective-proba001} 
 with the goal of obtaining a non-trivial uniform approximation valid for the starting point $x$ across the entire real line $\bb R$.
Our primary result demonstrates that the uniform asymptotic behaviour of these probabilities 
is subject to  new heat kernel type approximation. 
In addition, we provide a rate of convergence, ensuring a more precise understanding of the approximation's accuracy. 
Our findings cover both scenarios, where  
$x\in \bb R$ can be fixed and where $x$ depends on $n$, tending to infinity at a rate equal to $\sqrt{n}$ or faster. 
As a consequence, we extend and improve the results 
established separately for the two particular scenarios previously analyzed in Theorems 2.9 and 2.10 of \cite{GX-2024-AIHP}. 
Furthermore, the conditioned limit theorems established in this paper will be instrumental in developing a conditioned local limit theorem based on the heat kernel approximation, which will be considered in a subsequent paper.  

Our basic assumption is that the increment $X_1$ has moments of order $2+\delta$; specifically, we assume that   
there exists  a constant $\delta > 0$ such that 
\begin{align} \label{Moment condition 2+delta}
\bb E (|X_1|^{2+\delta})  < \infty.
\end{align} 
The limit behavior of the probabilities in \eqref{Objective-proba001} is tied to a harmonic function $V(x)$ 
associated with the random walk $(x+S_n )_{n\geq 1}$, which we proceed to introduce. 
It is known (see, for example, Theorem 2.1 and Lemma 7.1 in \cite{GLL18Ann})
 that under condition \eqref{Moment condition 2+delta}, the function
\begin{align}\label{Equ-Vx}
x \in \bb R \mapsto V(x) : = -\mathbb E S_{\tau_x} 
\end{align}
is well-defined, non-decreasing and non-negative; 
and it holds that for any $x \in \bb R$, 
\begin{align}\label{Def_V_002bb}
\lim_{n \to \infty } \mathbb{E} \left(x+S_n ; \tau_x  > n \right)  = -\mathbb E S_{\tau_x} = V(x),
\end{align}
where, for a random variable  $X$  and
 an event $B$, we write $\bb E (X; B)$ for the expectation $\bb E (X \mathds 1_{B})$. 
Using the Markov property and \eqref{Def_V_002bb}, one can deduce that, for any $x \in \bb R$,
\begin{align} \label{Doob transf}
\bb E \Big( V(x+S_1); x+S_1\geq 0  \Big) = V(x),   
\end{align} 
which means that the function $V$ is harmonic for the random walk $(x+S_n )_{n\geq 1}$ killed at the stopping time $\tau_x$. 
In particular, since $\bb E(X_1)=0$ and $\sigma^2>0$, 
it follows that $V(0)>0$ (so that $V$ is strictly positive on $\bb R_+$).
Moreover, the support of $V$ is given by (\cite[Example 2.10]{GLL18Ann}) 
\begin{align} \label{support-V}
\supp V = \{ x \in \bb R: \bb P(x + X_1 \geq 0) >0  \}. 
\end{align}

The function $V$ is related to the renewal function $U$  
in the strict descending ladder process of $(S_n)_{n\geq 1}$ as follows:   
for any  $x\geq 0$, it holds $U(x)=V(x)/V(0)$, 
where $V(0) = - \mathbb E S_{\tau_0}$ represents the expectation of the ladder height (see \cite{KV17}). 
Notably, while $U$ retains harmonic properties, it can be verified through renewal arguments 
under much weaker conditions than those required by the scenario studied in this paper. 
For further insights into this relationship, we refer to the works of 
 Tanaka \cite[Lemma 1]{Tanaka1999} and Kersting and Vatutin \cite[Lemma 4.2]{KV17}.

Let us end this section by recalling some more notation to be used all over the paper.
By $c$ we shall denote positive constants and by $c_{\alpha}, c_{\alpha, \beta}$ positive constants depending only on their indices.  
All these constants are supposed to be different at every occurrence.
We denote by $\mathds 1_{B}$ the indicator function of the set $B$. 
The standard normal distribution function is defined by $\Phi(x)=\frac{1}{\sqrt{2\pi}}\int_{-\infty}^x e^{-t^2/2} dt$, for $x\in \bb R$, 
and $\phi(x)=\frac{1}{\sqrt{2\pi}} e^{-x^2/2}$, $x\in\bb R$ stands for the standard normal density function.

We denote by $O(x)$ (respectively $o(x)$) as $x\to\infty$ a nonnegative function of $x$ such that 
$\limsup_{x \to \infty} O(x)/x \leq c < \infty$  (respectively $\lim_{x \to \infty} o(x)/x = 0$).
The notation $a_n \asymp b_n$ 
means that $0<c_1\leq  \liminf_{n \to \infty} a_n / b_n \leq \limsup_{n \to \infty} a_n / b_n  \leq c_2<\infty$. 
The notation $a_n(\alpha) \sim b_n(\alpha)$, uniformly in $\alpha \in A$ as $n\to \infty$, 
 means that $\lim_{n \to \infty} \sup_{\alpha \in A} | a_n(\alpha) / b_n(\alpha) - 1|=0$.

\subsection{Heat kernel asymptotic for the persistence probability}
The objective of this and the subsequent subsection is to derive asymptotic expressions for the persistence probability 
and to establish a central limit theorem for the random walk $(x+S_n )_{n\geq 1}$ 
conditioned to stay positive, with explicit convergence rates depending on the starting point $x$. 
As a direct application we will be able to  establish uniform local limit theorems for random walks conditioned to stay positive, 
which will be considered in a separate paper. 
Besides, our findings have intrinsic interest beyond this application. 
Recent results concerning the conditioned random walks can be found in \cite{Hong-Sun'24} for conditioning via Doob $h$-transforms 
and \cite{Sun 25} for moderate deviations. 
For related applications in the context of branching random walks, we refer to \cite{Blanchet-Zhang tight24, Hou-Ren-Song}.

We now introduce a function related to the Dirichlet heat kernel
 (as described in \eqref{H as int od Dirichlet kern-001} and \eqref{Def-Levydens} below): 
for $x\in \bb R$, 
\begin{align} \label{Heat func-001}
H(x)  = 2 \Phi(x) - 1, 
\end{align}
where $\Phi$ is the standard normal distribution function. 
In particular, for $x\geq 0$, this function has a probabilistic interpretation: $H(x) = \bb P(\tau_x^{B}>1),$ 
where $\tau_x^{B}$ is the exit time of the standard Brownian motion $(B_t)_{t\geq 0}$
from the half-line $\bb R_+$, starting from $x$.

Nagaev \cite{Nag69, Nag70} proved a Berry-Esseen type bound for the persistence probability 
$\mathbb{P} \left( \tau_x >n \right)$: 
under the assumption that $\beta_3=\bb E (|X_1|^3) < \infty$, 
uniformly in $x\geq 0$,
 \begin{align} \label{Nag-bound-001}
\left| \mathbb{P} \left( \tau_x >n \right)  -H\left(\frac{x}{\sigma\sqrt{n}}\right) \right| 
   \leq c (\beta_3)^2 n^{- 1/2  }.  
\end{align}
Aleskevicene \cite{Aleskev73} improved the upper bound in \eqref{Nag-bound-001} to $c \beta_3 n^{- 1/2  }$. 
 Note  that Nagaev's bound \eqref{Nag-bound-001} makes sense only when $x=x_n\to\infty$ as $n\to\infty$. 
For $x$ in compact sets it is not precise, 
since the remainder term $O(n^{-1/2})$ is of the same order as the main term 
$H(\frac{x}{\sigma\sqrt{n}})=2\Phi ( \frac{x}{\sigma \sqrt{n}} ) - 1$.
Actually, for any fixed $x\geq 0$, 
the right asymptotic of the persistence probability $\mathbb{P} \left( \tau_x >n \right)$
is not given any more by \eqref{Nag-bound-001}, but by the following equivalence result: as $n\to \infty$, 
\begin{align} \label{more familiar form 001}
\mathbb{P} \left( \tau_{x}>n\right) \sim \frac{2V(x)}{\sigma \sqrt{2\pi n}}. 
\end{align}
For $x=0$ the asymptotic \eqref{more familiar form 001} is known, for instance, from Feller \cite{Fel64}, Spitzer \cite{Spitzer}, 
Vatutin and Wachtel \cite{VatWacht09},
and for $x\geq 0$ with $\frac{x}{\sqrt{n}}\to 0$  from Doney \cite{Don12}.

We will establish a new asymptotic which unifies  \eqref{Nag-bound-001} and \eqref{more familiar form 001} 
as well as  provide a rate of convergence.  
As a consequence, we enhance the equivalence \eqref{more familiar form 001} by specifying the rate of convergence 
for $x$ near the boundary. We also refine Nagaev's bound \eqref{Nag-bound-001} 
by providing a more accurate estimate for large values of $x$.

\begin{figure}
  \includegraphics[width=10cm]{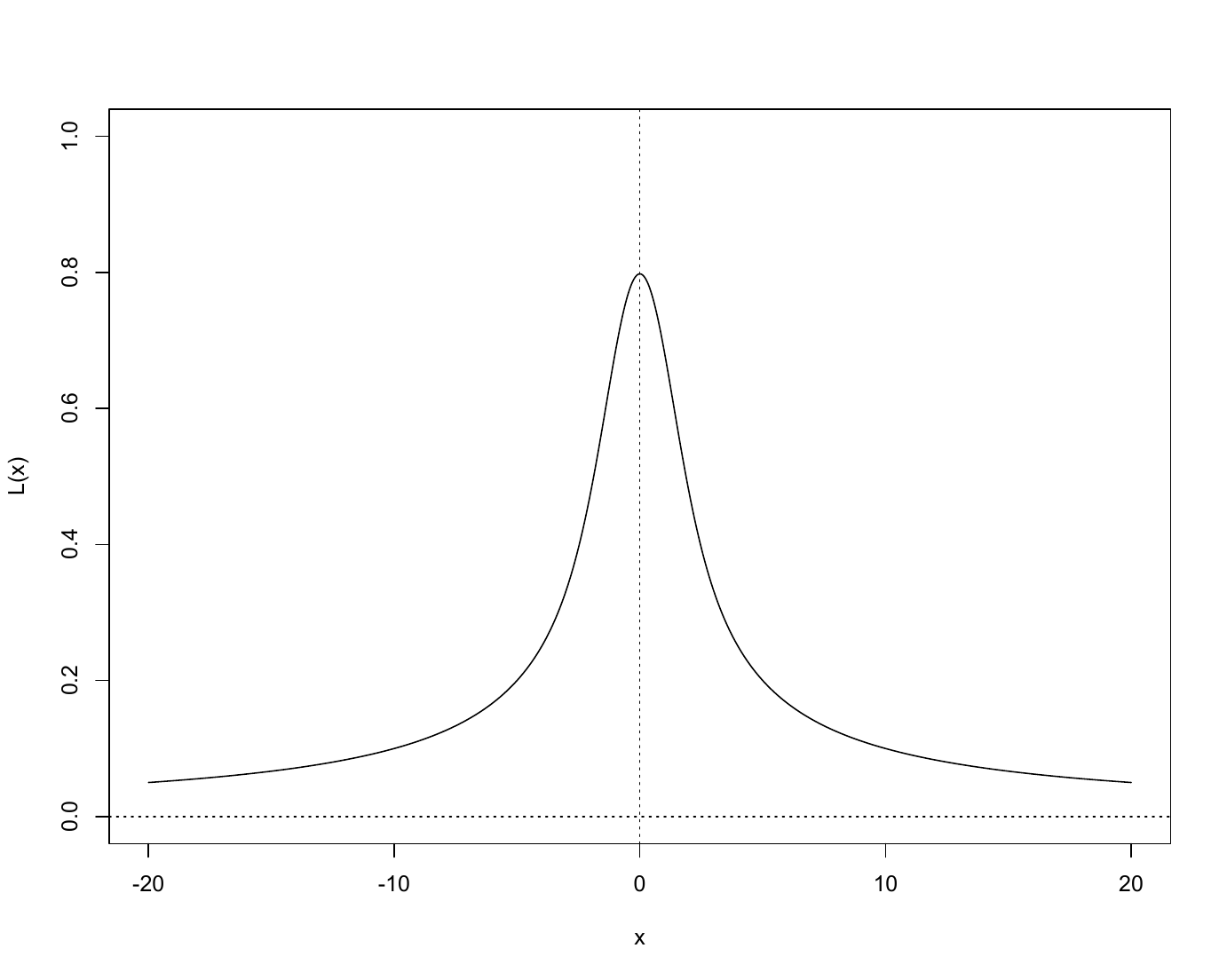}
  \caption{Plot of the function $x\to L(x)$.}
  \label{FigFuncL-001}
\end{figure}
In order to state our results 
we introduce the function:
for $x\in \bb R$,  
\begin{align} \label{def-Lx-001}
L(x)= \frac{H(x)}{x} = \frac{2\Phi(x)-1}{x},
\end{align}
where at $x=0$ we define $L$ by continuity, 
\begin{align} \label{value of L(0)}
L(0) := \lim_{x\to 0} L(x) = \frac{2}{\sqrt{2\pi}}.
\end{align}
Note that, $L$ is positive and even on $\bb R$, i.e.\ $L(x)>0$ and $L(x)=L(-x)$, for any $x\in \bb R$. 
Also, the function $x\mapsto L(x)$ is decreasing on $\bb R_+$ and satisfies, as $|x|\to\infty$, 
\begin{align*} 
L(x) \sim |x|^{-1}.
\end{align*}
The important Lipschitz property of $L$ is established in Lemma \ref{lem-inequality for L} in the next section.
This provides a controlled rate of change in $L$, which is critical for analyzing asymptotic probabilities. 
An alternative representation, connecting $L$ to the heat kernel to be introduced below,
is given by \eqref{func L as integral of heat kern}.
The plot of the function $x\mapsto L(x)$, depicted in Fig. \ref{FigFuncL-001}, illustrates its symmetry and positivity. 

 Our main result regarding the persistence probability $\bb P(\tau_x >n)$ is expressed in relation 
to the harmonic function $V$ and the function $L$. 
This unified framework not only provides an effective approximation of $\bb P(\tau_x >n)$
 but also highlights the interplay between the harmonic and heat kernel representations.

\begin{theorem}\label{Theor-probtauUN-001} 
Assume that $\bb E (X_1) = 0$, $\bb E (X^2_1)= \sigma^2  > 0$  and that there exists $\delta > 0$ 
such that $\bb E (|X_1|^{2+\delta})  < \infty.$ 
Then there exists a constant $c >0$ such that for any  $n \geq 1$ and $x\in \bb R$,   
\begin{align} \label{UniformCondtau-001}
 \left| \mathbb{P} \left( \tau_{x}>n\right) - \frac{V(x)L\left(  \frac{x}{\sigma \sqrt{n}} \right) }{\sigma \sqrt{n}} \right| 
\leq  
\frac{c}{\sqrt{n}} \left(\frac{ 1}{ n^{\frac{\delta}{4} }} 
+ \frac{V(x)L\left(  \frac{x}{\sigma \sqrt{n}} \right) }{n^{\frac{\delta}{4(3+\delta)} } }  \log n\right) := \frac{c}{\sqrt{n}} R_n(\delta,x).
\end{align}
\end{theorem}

The proof of Theorem \ref{Theor-probtauUN-001} will be presented in Section \ref{SecProof Theor-probtau-001}.
The asymptotic provided by Theorem \ref{Theor-probtauUN-001} 
is non-trivial for any $x\in \supp V$. 
This is ensured through the following bound, which follows directly 
from the properties of the functions $V$ and $L$ (see Lemma \ref{Lem-properties-of-V-L} below):
for any  $a\in \supp V$, there exist constants $c_1, c_2 >0$ such that, for any $n\geq 1$ and $x\in [a,\infty)$,
\begin{align} \label{lower bound of VL-001}
c_1 \leq V(x)L\left(  \frac{x}{\sigma \sqrt{n}} \right) \leq  c_2 \sqrt{n}.
\end{align}
The bound \eqref{lower bound of VL-001} demonstrates the uniform behavior of $V_{\sigma^2 n}(x)$ 
across different regimes of $x$. 
This property is useful, for example, to deduce the following equivalence.
\begin{corollary}\label{Coroll-probtauUN-001b} 
Under the assumptions of Theorem \ref{Theor-probtauUN-001},
 for any $a\in \supp V$, it holds that, as $n\to\infty$, uniformly over $x\in [a,\infty)$,  
\begin{align} \label{UniformCondtau-001a}
 \mathbb{P} \left( \tau_{x}>n\right) \sim \frac{V(x)L\left(  \frac{x}{\sigma \sqrt{n}} \right) }{\sigma \sqrt{n}}.
\end{align}
In particular, as $V(0)>0$, the equivalence \eqref{UniformCondtau-001a} holds true uniformly for $x\in [0,\infty)$. 
\end{corollary}

From \eqref{UniformCondtau-001a} and \eqref{lower bound of VL-001}, 
we obtain the following two-sided bound for the probability $\mathbb{P} \left( \tau_{x}>n\right)$: 
for any $a\in \supp V$, there exist constants $c_1, c_2 >0$ such that, for any $n\geq 1$ and $x\in [a,\infty)$, 
\begin{align} \label{UniformCondtau-001bac}
c_1 \frac{V(x)L\left(  \frac{x}{\sigma \sqrt{n}} \right) }{\sqrt{n}} \leq  \mathbb{P} \left( \tau_{x}>n\right) 
 \leq  c_2 \frac{V(x)L\left(  \frac{x}{\sigma \sqrt{n}} \right) }{\sqrt{n}}. 
\end{align}
The results \eqref{UniformCondtau-001a} and \eqref{UniformCondtau-001bac} are new in the presented ranges. 
A two-sided bound without any constraint on $x$ can be easily deduced from Theorem \ref{Theor-probtauUN-001}. 


When $\delta \geq 1$, 
Theorem \ref{Theor-probtauUN-001} can be slightly improved using Nagaev's bound \eqref{Nag-bound-001}.
\begin{theorem}\label{Theor-probtauUN-001c} 
Assume that $\bb E (X_1) = 0$, $\bb E (X^2_1)= \sigma^2  > 0$  and that there exists $\delta \geq 1$ 
such that $\bb E (|X_1|^{2+\delta})  <  \infty.$ 
Then, there exists a constant $c >0$ such that for any  $n \geq 1$ and $x\in \bb R$,  
\begin{align} \label{Theor-probtauUN-002}
 \left| \mathbb{P} \left( \tau_{x}>n\right) - \frac{V(x)L\left(  \frac{x}{\sigma \sqrt{n}} \right) }{\sigma \sqrt{n}}  \right| 
\leq  \frac{c}{\sqrt{n}} \left(\frac{ 1}{ n^{\frac{\delta}{4} }} + \frac{V(x) L\left(\frac{x}{\sigma\sqrt{n}} \right) }{n^{\frac{1}{4}} }  \log n\right). 
\end{align}
\end{theorem}

Theorems \ref{Theor-probtauUN-001} and \ref{Theor-probtauUN-001c}  can be used to deduce asymptotics in the following cases: 
$x$ close to the origin on the $\sqrt{n}$-scale, and $x$ away from the origin on the $\sqrt{n}$-scale. 
Below we deal only with Theorem \ref{Theor-probtauUN-001},  the results for Theorem \ref{Theor-probtauUN-001c}  being similar.
  
A1. In the case when $\frac{x}{\sqrt{n}}\to 0$, we can recover the previously known result \eqref{more familiar form 001} and
additionally provide a rate of convergence. 
Indeed, using the fact that $|L(u) - L(0)| \leq c |u|$ for any  $u\in \bb R$ (by Lemma \ref{lem-inequality for L}), 
from Theorem \ref{Theor-probtauUN-001} it follows that, for any sequence $\alpha_n\to 0$ as $n\to\infty$, 
uniformly in $n \geq 1$ and $x\in \bb R$ satisfying $\frac{|x|}{\sqrt{n}} \leq \alpha_n $, 
\begin{align} \label{part-form 001}
\Bigg| \mathbb{P} \left( \tau_{x}>n\right) - \frac{2V(x)}{\sigma \sqrt{2\pi n}} \Bigg| 
\leq 
c\frac{V(x)}{\sqrt{2\pi n}} \alpha_n   
+ \frac{c}{\sqrt{n}}\left(\frac{ 1}{ n^{\frac{\delta}{4} }} 
+ \frac{V(x)  }{n^{\frac{\delta}{4(3+\delta)} } }  \log n\right).
\end{align}
In particular, if $a\in \supp V$, from \eqref{part-form 001} 
we get that \eqref{more familiar form 001} holds uniformly in $x\in [a,\alpha_n \sqrt{n}]$. 
Under much more restrictive assumptions,  a higher order expansion of the probability $\mathbb{P} \left( \tau_{x}>n\right)$ for $x\geq 0$ has been recently obtained in \cite{DTW24}.

A2. To analyse the behaviour on the $\sqrt{n}$-scale, we let  $x=t\sigma\sqrt{n}$ for $t \in \bb R$.
Taking into account \eqref{def-Lx-001}, it holds that, for $x\in \bb R$,
\begin{align} \label{main term identity-001}
\frac{V(x) L\left(\frac{x}{\sigma\sqrt{n}}\right) }{\sigma\sqrt{n}} = \frac{V(x)}{x} H\left(\frac{x}{\sigma\sqrt{n}}\right), 
\end{align}
where, for $x =0$, the identity holds true by virtue of \eqref{def-Lx-001} and \eqref{value of L(0)}.
This, together with \eqref{UniformCondtau-001} and Lemma \ref{Lem-bound for V-001}, gives 
\begin{align} \label{rc-004a}
\left| \mathbb{P} \left( \tau_x >n \right)  -H\left(\frac{x}{\sigma\sqrt{n}}\right) \right|
& \leq   c\left(  \left(\frac{V(x)}{x} - 1\right) H(t)
+ n^{-\frac{1}{2} -  \frac{\delta}{4} } + L(t) n^{-\frac{\delta}{4(3+\delta)}}\log n\right) \notag\\
& \leq c\left(  n^{- \delta/2} t^{-\delta} H\left( t \right)
+ n^{-\frac{1}{2} -  \frac{\delta}{4} } + L(t) n^{-\frac{\delta}{4(3+\delta)}}\log n\right). 
\end{align}
In particular,  the result \eqref{rc-004a} complements Nagaev's bound \eqref{Nag-bound-001} for $\delta \in (0, 1)$.

\subsection{A conditioned heat kernel limit theorem}

We now turn to the asymptotic of the joint probability $\mathbb{P} \left(  \frac{ x + S_n }{\sigma\sqrt{n}} \leq  u,  \tau_x >n\right)$.
To state the corresponding result we shall make use of the function
\begin{align} \label{Def-Levydens}
\psi(x,y) = \phi(x-y) -\phi(x+y)
= \frac{1}{\sqrt{2\pi }}   \left( e^{-\frac{(x-y)^2}{2}} - e^{-\frac{(x+y)^2}{2}} \right),  \quad  x, y  \in \bb R,   
\end{align} 
which is called the one-dimensional Dirichlet heat kernel. 
It is easy to see that $\psi$ is symmetric on $\bb R \times \bb R$, i.e. $\psi(x, y) = \psi(y, x)$ for any $x, y \in \bb R$, 
and nonnegative on $\bb R_+ \times \bb R_+$. 
Also the map $y\mapsto \psi(x, y)$ is odd, i.e. $\psi(x,-y) = -\psi(x, y)$ for any $x, y \in \bb R$.
The relation of the heat kernel $\psi$ to the standard Brownian motion 
is given in Section \ref{sec-properties heat kern}. 
In particular, from Lemma \ref{lemma tauBM}, for any $x\in \bb R$, it holds
\begin{align} \label{H as int od Dirichlet kern-001}
H(x) = \int_{0}^{\infty} \psi(x,y) dy.
\end{align}
In addition to the function $\psi$, we need its normalized version which is defined as follows: for any $x\in \bb R$,
\begin{align} \label{def of func h-001}
\ell_H(x,y): = \frac{\psi(x,y)}{H(x)} = \frac{\psi(x,y)}{2\Phi(x)-1},\quad y\in \bb R. 
\end{align}
For $x=0$, by continuity we have, for any $y\in \bb R$, 
\begin{align} \label{approx by Rayleigh density-001} 
\ell_H(0,y) := \lim_{x\to 0} \ell_H(x,y)  = \phi^+(y),
\end{align}
where  
\begin{align} \label{Def-Rayleighdens}
\phi^+(y) := y e^{-y^{2}/2}, \quad y \in \bb R.   
\end{align} 
The restriction of the function $y\mapsto \phi^+(y)$ to $\bb R_+$ is known as the Rayleigh density function.
Note that the restriction of $(x,y)\mapsto \ell_H(x,y)$ on $\bb R \times \bb R_+$ is a Markov kernel, which means that 
for any $x\in \bb R$, the function $y\mapsto \ell_H(x,y)$ is a density function on $\bb R_+$;
in particular, it is positive on $\bb R_+\setminus\{ 0\}$. 
The plot of the function $(x,y) \mapsto \ell_H(x,y)$ is given in Fig. \ref{FigFunc-h-002}. 
\begin{figure}
  \includegraphics[width=10cm]{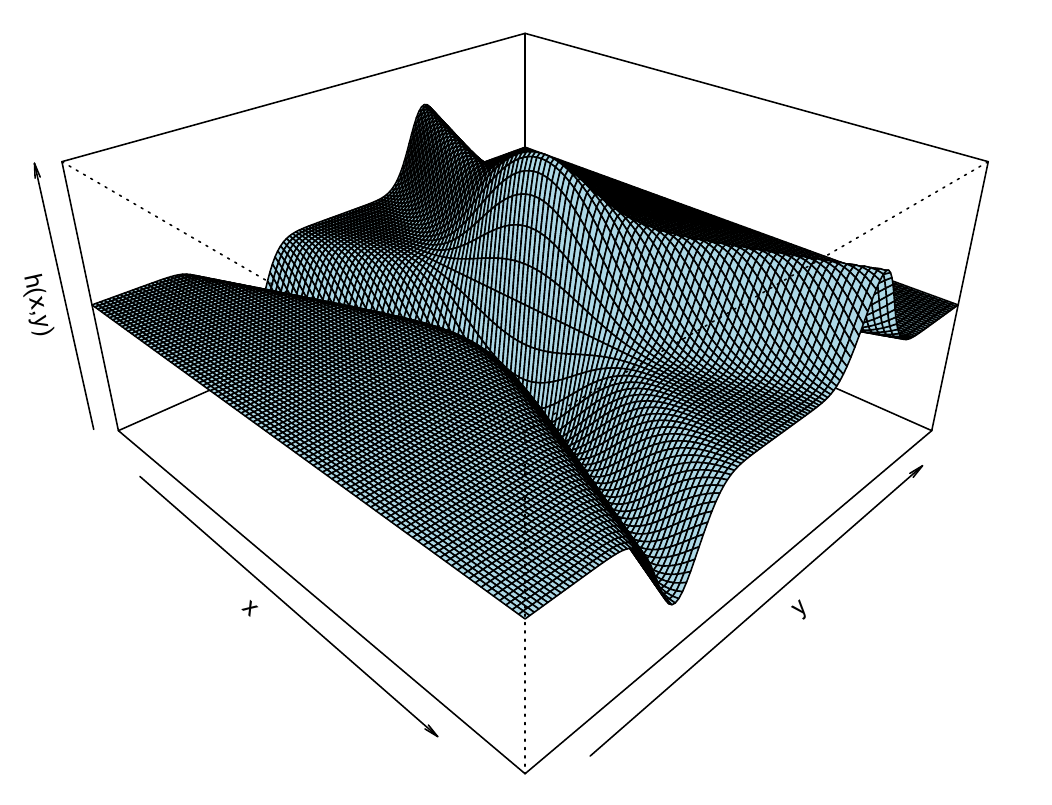}
  \caption{Plot of the function $(x,y)\mapsto \ell_H(x,y)$.}
  \label{FigFunc-h-002}
\end{figure}

Our second main result is the following Berry-Esseen type bound for the random walk $x+S_n$ conditioned to stay 
in the half-line $\bb R_+$.  

\begin{theorem}\label{Theor-probIntUN-002} 
Assume that $\bb E (X_1) = 0$, $\bb E (X^2_1)= \sigma^2  > 0$  and that there exists $\delta > 0$ 
such that $\bb E (|X_1|^{2+\delta})  < \infty.$ 
Then, there exists a constant $c >0$ such that, for any $n \geq 1$, $u\geq 0$ and $x\in \bb R$, 
\begin{align} \label{UniformCondInt-001}
\Bigg| \mathbb{P} \Bigg(  \frac{x+S_n }{ \sigma \sqrt{n}} \leq  u, \tau_x >n\Bigg) 
  -\frac{V(x)L\left(  \frac{x}{\sigma \sqrt{n}} \right) }{\sigma \sqrt{n}}  \int_{0}^{u}  \ell_H \left( \frac{x}{\sigma \sqrt{n}}, y   \right)  dy  \Bigg| 
       \leq c \frac{R_n(\delta,x)}{\sqrt{n}},
\end{align}
where $R_n(\delta,x)$ is defined by \eqref{UniformCondtau-001}. 
\end{theorem}
For the proof we refer to Section \ref{sec Theor-probIntUN-002}. 
As a direct consequence of Theorem \ref{Theor-probIntUN-002} and of Corollary \ref{Coroll-probtauUN-001b}, we obtain the following generalization of the central limit theorem to conditioned random walks.

\begin{corollary}\label{Coroll-probtauUN-002b}  
Assume conditions of Theorem \ref{Theor-probIntUN-002}. 
For any $a\in \supp V$, it holds 
\begin{align} \label{UniformCondtau-002b}
\lim_{n\to\infty} \sup_{x \geq a} \sup_{u \geq 0} \Bigg|\mathbb{P} \Bigg(  \frac{x+S_n }{ \sigma \sqrt{n}} \leq  u \bigg| \tau_x >n\Bigg)  
-  \int_{0}^{u}  \ell_H \left(\frac{x}{\sigma\sqrt{n}} , y  \right)dy \Bigg| =0. 
\end{align}
In particular, as $V(0)>0$, \eqref{UniformCondtau-002b} holds true with $a = 0$.
\end{corollary}


On the $\sqrt{n}$-scale the approximation \eqref{UniformCondtau-002b} can be reformulated as follows: 
uniformly over $u\geq 0$ and $t\in [0,\infty)$, 
\begin{align*} 
\lim_{n\to\infty}\mathbb{P} \Bigg(  t+\frac{S_n }{ \sigma \sqrt{n}} \leq  u \bigg| \tau_{t\sigma\sqrt{n}} >n\Bigg)  
=  \mathscr L_H \left( t,u  \right) :=\int_{0}^{u}  \ell_H(t,y)dy,
\end{align*}
where, for any $t\in [0,\infty)$ fixed, $u\mapsto \mathscr L_H \left( t,u  \right)$ is the probability distribution with density $y\mapsto \ell_H(t,y)$ on $\bb R_+$. 
In particular, for $t=0$, by \eqref{approx by Rayleigh density-001}, we have that, uniformly for $u\geq 0$,
\begin{align*} 
\lim_{n\to\infty}\mathbb{P} \Bigg(  \frac{S_n }{ \sigma \sqrt{n}} \leq  u \bigg| \tau_{0} >n\Bigg)  
= \Phi^+(u):= \int_0^{u}\phi^+(y)dy, 
\end{align*}
where $\Phi^+$ is the Rayleigh distribution function.

As in the case of probability $\mathbb{P} \left( \tau_{x} > n\right)$,
 we shall analyse the scenarios when $x$ is close to the origin 
 and away from the origin on the $\sqrt{n}$-scale.

B1. First we consider the case when $\frac{x}{\sqrt{n}}\to 0$.
Using the approximation of $\ell_H(x,\cdot)$ by the Rayleigh law given in \eqref{approx by Rayleigh density-001}
and Lemma \ref{Holder prop for int ell}, 
from Theorem \ref{Theor-probIntUN-002} we can deduce that, there exists a constant $c>0$ such that, for any sequence
  $\alpha_n\to 0$ as $n\to\infty$, for any $n \geq 1$, $u\geq 0$ and $x\in \bb R$ satisfying $\frac{|x|}{\sqrt{n}} \leq \alpha_n $,
\begin{align} \label{intro200-1}
\Bigg| \mathbb{P} \Bigg(  \frac{x+S_n }{ \sigma \sqrt{n}} \leq  u, \tau_x >n\Bigg) 
    - \frac{2V(x) }{\sigma \sqrt{2\pi n}}     
 \int_{0}^{u}  \phi^+(y)  dy  \Bigg| 
\leq c\frac{V(x)}{\sqrt{ n}} \alpha_n   + c \frac{R_n(\delta,x)}{\sqrt{n}}.
\end{align}
In particular, for any $\ee>0$ and $a\in \supp V$, from \eqref{intro200-1} 
we have as $n\to\infty$, uniformly in $u\geq \ee$ and $x\in [a,\infty)$ satisfying $\frac{|x|}{\sqrt{n}} \leq \alpha_n $,
\begin{align} \label{intro200-2}
\mathbb{P} \Bigg(  \frac{x+S_n }{ \sigma \sqrt{n}} \leq  u, \tau_x >n\Bigg) 
     \sim 
     \frac{2V(x) }{\sigma \sqrt{2\pi n}}  \int_{0}^{u}  \phi^+(y) dy. 
\end{align}

B2. Let  $x=t\sigma\sqrt{n}$, where $t >0$. 
Taking into account \eqref{main term identity-001} 
and Lemma \ref{Lem-bound for V-001}, from \eqref{UniformCondInt-001} we get the following  rate of convergence: 
\begin{align} \label{intro200-3}
& \left|  \mathbb{P} \left(  \frac{x+S_n }{ \sigma \sqrt{n}} \leq  u, \tau_x >n\right) - \int_{0}^{u}  \psi \left( t , y   \right)  dy  \right| \notag\\
&\qquad\qquad \leq   c\left(  n^{- \delta/2} t^{-\delta} H\left( t \right)
+ n^{-\frac{1}{2} -  \frac{\delta}{4} } + L(t) n^{-\frac{\delta}{4(3+\delta)}}\log n\right).
\end{align}

Borovkov \cite[Theorem 6]{Borovk62} 
obtained precise asymptotics of any order of the persistence probability $\mathbb{P} \left( \tau_{x}>n\right)$ and of the large
deviation probability
$\mathbb{P} \big(  \frac{x+S_n }{ \sigma \sqrt{n}} >  u, \tau_x >n\big)$.  
However, the results in \cite{Borovk62} do not apply when $x$ is fixed and are established
under much stronger conditions than ours, namely that $X_1$ has an exponential moment
and that the distribution of $X_1$ has an absolutely continuous component.

We end this section by giving an alternative expression for the main term in Theorem \ref{Theor-probIntUN-002}. 
Consider the function $\ell_h$ defined as follows:  for any $x\in \bb R$,
\begin{align} \label{def of func ell_h-001}
\ell_h(x,y) = \frac{\psi(x,y)}{h(x)}, \quad y\in \bb R, 
\end{align}
where $h(x)=x$ is the so called profile harmonic function. 
For $x=0$ the function $\ell_h$ is defined by continuity, i.e. for any $y\in \bb R$, 
\begin{align} \label{ell_h is approx by Rayleigh density-001} 
\ell_h(0,y) := \lim_{x\to 0} \ell_h(x,y)  = L(0) \phi^+(y). 
 \end{align}
 Moreover, for any $x\in\bb R$,
\begin{align} \label{func L as integral of heat kern}
L(x) =\int_{\bb R_+} \ell_h(x,y)dy.
\end{align}
With these notation, one can rewrite the main term of \eqref{UniformCondInt-001} as
\begin{align*} 
V(x) L\left(\frac{x}{\sigma\sqrt{n}}\right)  \int_{0}^{u}  \ell_H \left( \frac{x}{\sigma \sqrt{n}}, y   \right)  dy
=V(x)  \int_{0}^{u}  \ell_h \left( \frac{x}{\sigma \sqrt{n}}, y   \right)  dy. 
\end{align*}


\section{Preliminary statements}\label{Sec-Auxiliary statements}

\subsection{Properties of the heat kernel} \label{sec-properties heat kern}

Let $\left( B_{t}\right)_{t\geq 0}$ be a standard Brownian motion on the
probability space $\left(\Omega, \mathscr{F}, \mathbb{P} \right).$ 
For any $x \geq 0$, define the exit time 
\begin{align*}
\tau_{x}^{\sigma B} = \inf \left\{ t \geq 0: x + \sigma B_{t} < 0 \right\}. 
\end{align*}
The following well known formulas are due to Levy \cite{Levy37} (Theorem 42.I, pp.194-195).

\begin{lemma}\label{lemma tauBM} 
For any $x \geq 0$, $n\geq 1$ and $0 \leq a < b \leq \infty$, it holds  
\begin{align*}
\bb P \left( x +  \sigma B_{n} \in \left[a, b\right],  \tau_{x}^{\sigma B}>n \right) 
=  \frac{1}{\sigma \sqrt{ n} } \int_{a}^{b}  \psi \left( \frac{x}{\sigma \sqrt{n}},\frac{y}{\sigma \sqrt{n}} \right) dy. 
\end{align*}
In particular, by taking $a = 0$ and $b = \infty$, for any $x \geq 0$ and $n\geq 1$, 
\begin{align*}
\bb P \left( \tau_{x}^{\sigma B} > n \right) 
= 2 \Phi \left( \frac{x}{\sigma\sqrt{n}} \right)-1 
= \frac{2}{\sigma \sqrt{2\pi n} } \int_{0}^x  e^{ - \frac{s^{2}}{2 \sigma^2 n} } ds.   
\end{align*}
\end{lemma}

When $\sigma=1$ and $n=1$, we have, for $x \geq 0$, 
\begin{align} \label{def-Hx-001}
H(x) = \bb P \left( \tau_{x}^{B} > 1 \right) 
= 2 \Phi \left( x \right)-1
= \frac{2}{\sqrt{2\pi} } \int_{0}^x  e^{ - \frac{s^{2}}{2} } ds. 
\end{align}

\begin{lemma} \label{lem-inequality for H}
For any $x\geq 0$ and $\ee\geq0$, we have 
\begin{align} \label{bound for H-001}
H(x(1+\ee)) \leq H(x) (1+\ee)
\end{align}
and 
\begin{align} \label{bound for H-002}
H(x(1-\ee)) \geq H(x) (1-\ee).
\end{align}
\end{lemma}

\begin{proof} 
Since $H(0) = 0$, 
the first inequality \eqref{bound for H-001} 
is a consequence of the concavity of the map $x\mapsto H(x)$ for $x\in \bb R_+$. 
Similarly, the second inequality \eqref{bound for H-002} 
follows from the convexity of the map $x\mapsto H(x)$ for $x\in \bb R_-$.  
\end{proof}

We now establish the important Lipschitz property of the function $L$,
which is crucial in the proof of Theorem \ref{Theor-probtauUN-001}.
 
\begin{lemma} \label{lem-inequality for L}
There exists a constant $c>0$ such that for any $x, x' \in \bb R$, 
\begin{align*} 
\left| \frac{L(x')}{L(x)} -1 \right| \leq c | x' - x |.
\end{align*}
 \end{lemma}

\begin{proof}
For any $x, x' \in \bb R$, denote $a= x' - x$ and 
\begin{align*} 
R(x,a) : = \frac{L(x+a)-L(x)}{aL(x)}.
\end{align*}
An elementary analysis which is performed below shows that $\sup_{x,a\in \bb R}|R(x,a)| \leq c$, for some positive constant $c > 0$. 
Therefore, the result follows.

Since the function $L$ is even, we have $R(x, a) = - R(-x, -a)$ for any $x, a \in \bb R$. 
Therefore, it suffices to show that $|R(x,a)| \leq c$ in the following five cases:

Case (i): $x \geq 0$ and $a \geq 0$. 
By the definition of $H$ and $L$ (cf.\ \eqref{def-Hx-001} and \eqref{def-Lx-001}), we have 
\begin{align*}
I(x, a) 
:= \frac{H(x+a)-H(x)}{a L(x)}
=  \frac{ \frac{1}{a} \int_{x}^{x+a}  e^{ - \frac{s^{2}}{2} } ds }{ \frac{1}{x} \int_{0}^{x}  e^{ - \frac{s^{2}}{2} } ds }. 
\end{align*}
Since the function $s \mapsto e^{ - \frac{s^{2}}{2} }$ is decreasing on $\bb R_+$, 
for any $x \geq 0$ and $a \geq 0$,  
we have $e^{ - \frac{(x+a)^2}{2} } \leq \frac{1}{a} \int_{x}^{x+a}  e^{ - \frac{s^{2}}{2} } ds \leq e^{ - \frac{x^{2}}{2} }$
and $e^{ - \frac{x^{2}}{2} } \leq \frac{1}{x} \int_{0}^{x}  e^{ - \frac{s^{2}}{2} } ds \leq 1$,  
so that $e^{ - \frac{(x+a)^2}{2} } \leq  I(x, a) \leq 1$. 
Hence, for any $x \geq 0$ and $a \geq 0$,  
\begin{align*}
\frac{1}{x+a} \left( e^{ - \frac{(x+a)^2}{2} } - 1  \right) \leq R(x,a) 
= \frac{1}{x+a} \left( I(x, a)  - 1 \right) \leq 0. 
\end{align*}
Let $f(t) = \frac{1}{t} ( e^{ - \frac{t^2}{2} } - 1 )$, $t \geq 0$. 
For $t \geq 1$, 
we have $f(t) \geq -\frac{1}{t} \geq -1$. 
For $t \in [0, 1)$, using the inequality $e^{ - \frac{t^2}{2} } \geq 1 - \frac{t^2}{2}$, 
we get $f(t) \geq - \frac{t}{2} \geq - \frac{1}{2}$. 
Therefore, $\frac{1}{x+a} ( e^{ - \frac{(x+a)^2}{2} } - 1 ) \geq -1$ for all $x, a \geq 0$
and we have proved that $\sup_{x, a \geq 0} |R(x,a)| \leq 1$.

Case (ii): $2 \geq x \geq 0$ and $-2 \leq a < 0$. 
By the mean value theorem, there exists $\tilde x \in (x+a, x)$ such that 
\begin{align*} 
L(x+a) - L(x) = a L'(\tilde x)  = a \frac{L'(\tilde x)}{L(\tilde x)} L(\tilde x). 
\end{align*}
Since $L(x) = \frac{H(x)}{x}$, we have 
\begin{align*}
\frac{L'(x)}{L(x)} = - \frac{1}{x} + \frac{H'(x)}{H(x)} =  - \frac{1}{x} + \frac{ e^{- \frac{x^2}{2}} }{ \int_0^x  e^{- \frac{s^2}{2}} ds }.
\end{align*}
By L'H\^opital's rule, it holds that $\lim_{x \to 0}  \frac{L'(x)}{L(x)} = 0$. 
Hence the function $x \mapsto \frac{L'(x)}{L(x)}$ is continuous on $\bb R$ 
and bounded on $[-2, 2]$, 
so $| \frac{L'(\tilde x)}{L(\tilde x)} | \leq c$ for some constant $c$. 
As the function $L$ is continuous and strictly positive on $[-2, 2]$, 
there exists a constant $c>0$ such that $L(\tilde x) \leq c L(x)$ for all $x, \tilde x \in [-2, 2]$.
Therefore, there exists a constant $c'>0$ such that $|L(x+a) - L(x)| \leq c' |a| L(x)$ for all 
$2 \geq x \geq 0$ and $-2 \leq a < 0$, as desired.

Case (iii): $x \geq 1$, $a <0$ and $x + a \geq 1$.  
Since the function $s \mapsto e^{ - \frac{s^{2}}{2} }$ is decreasing on $\bb R_+$, we have 
$\frac{1}{x} \int_{0}^{x}  e^{ - \frac{s^{2}}{2} } ds \leq \frac{1}{x+a} \int_0^{x+a} e^{ - \frac{s^{2}}{2} } ds$. 
Hence, 
\begin{align}\label{upper-I-x-a}
I(x, a) 
= \frac{ \frac{1}{a} \int_{x}^{x+a}  e^{ - \frac{s^{2}}{2} } ds }{ \frac{1}{x} \int_{0}^{x}  e^{ - \frac{s^{2}}{2} } ds }
= \frac{ \frac{1}{-a} \int_{x+a}^{x}  e^{ - \frac{s^{2}}{2} } ds }{ \frac{1}{x} \int_{0}^{x}  e^{ - \frac{s^{2}}{2} } ds }
\leq 1, 
\end{align}
so that $R(x, a) = \frac{1}{x+a} \left( I(x, a)  - 1 \right) \leq 0$. 
For the lower bound, 
since $\frac{1}{-a} \int_{x+a}^{x}  e^{ - \frac{s^{2}}{2} } ds \geq e^{ - \frac{x^2}{2} }$
and $\frac{1}{x} \int_{0}^{x}  e^{ - \frac{s^{2}}{2} } ds \leq 1$, 
we have $I(x, a) \geq e^{ - \frac{x^2}{2} }$. 
Taking into account that $x + a \geq 1$, we get 
\begin{align*}
R(x,a) \geq  \frac{1}{x+a} \left( e^{ - \frac{x^2}{2} } - 1  \right) \geq - \frac{1}{x + a} \geq -1. 
\end{align*}
Therefore, $|R(x, a)| \leq 1$ for all $x \geq 1$ and $a <0$ with $x + a \geq 1$.

Case (iv): $x \geq 1$, $a <0$ and $x+a \in [0, 1]$. 
As in \eqref{upper-I-x-a}, we have $R(x, a) \leq 0$ in this case, so it remains to give a lower bound for $R(x, a)$. 
Since $\frac{1}{-a} \int_{x+a}^{x}  e^{ - \frac{s^{2}}{2} } ds \geq e^{ - \frac{x^2}{2} }$
and $\frac{1}{x} \int_{0}^{x}  e^{ - \frac{s^{2}}{2} } ds \leq 1$, 
we get
\begin{align*}
R(x,a) =
 \frac{1}{x+a} \left( \frac{ \frac{1}{-a} \int_{x+a}^x e^{- \frac{s^2}{2}} ds }{ \frac{1}{x} \int_0^x e^{- \frac{s^2}{2}} ds } - 1\right)
\geq  \frac{1}{x+a} \left( e^{- \frac{x^2}{2}} - 1 \right). 
\end{align*}
Now we further consider two cases: $x+2a >0$ and $x+2a \leq 0$. 
If $x+2a >0$, we have $\frac{1}{x+a} \leq \frac{2}{x}$.
Since in case (i) we have shown that 
$f(t) = \frac{1}{t} ( e^{ - \frac{t^2}{2} } - 1 ) \geq -1$ for all $t \in \bb R_+$, we get 
\begin{align*}
R(x,a) \geq  \frac{1}{x+a} \left( e^{- \frac{x^2}{2}} - 1 \right) \geq \frac{2}{x} \left( e^{- \frac{x^2}{2}} - 1 \right) \geq -2. 
\end{align*}
If $x + 2a \leq 0$, we have $\frac{x}{a} \geq -2$. 
Since $a <0$ and $x+a \geq 0$, it holds that $\frac{1}{x} \leq \frac{1}{-a}$ and hence 
\begin{align*}
\frac{1}{-a} \int_{x+a}^x e^{- \frac{s^2}{2}} ds - \frac{1}{x} \int_0^x e^{- \frac{s^2}{2}} ds
\geq  \frac{1}{-a} \left(  \int_{x+a}^x e^{- \frac{s^2}{2}} ds -  \int_0^x e^{- \frac{s^2}{2}} ds \right)
= \frac{1}{a} \int_0^{x+a} e^{- \frac{s^2}{2}} ds. 
\end{align*}
Note that $\frac{1}{x+a} \int_0^{x+a} e^{- \frac{s^2}{2}} ds \leq 1$, 
and $\int_0^x e^{- \frac{s^2}{2}} ds \geq \int_0^1 e^{- \frac{s^2}{2}} ds \geq e^{-1/2}$
since $x \geq 1$. 
Therefore, taking into account that $\frac{x}{a} \geq -2$, we obtain 
\begin{align*}
R(x,a) = \frac{1}{x+a}  \frac{ \frac{1}{-a} \int_{x+a}^x e^{- \frac{s^2}{2}} ds - \frac{1}{x} \int_0^x e^{- \frac{s^2}{2}} ds
}{ \frac{1}{x} \int_0^x e^{- \frac{s^2}{2}} ds }
& \geq \frac{1}{x+a}  \frac{ \frac{1}{a} \int_0^{x+a} e^{- \frac{s^2}{2}} ds }{ \frac{1}{x} \int_0^x e^{- \frac{s^2}{2}} ds } \notag\\
& = \frac{x}{a}  \frac{\frac{1}{x+a} \int_0^{x+a} e^{- \frac{s^2}{2}} ds}{\int_0^x e^{- \frac{s^2}{2}} ds}
\geq -2 e^{-1/2}. 
\end{align*}
Combining the above estimates, we conclude that $|R(x, a)| \leq 2$ for all $x \geq 1$ and $a <0$ with $x+a \in [0, 1]$.

Case (v): $1 \geq x > 0$ and $a <-2$. 
In this case, we have $e^{- \frac{(x+a)^2}{2} } \leq \frac{1}{-a} \int_{x+a}^{x}  e^{ - \frac{s^{2}}{2} } ds \leq 1$
and $e^{ - \frac{x^{2}}{2} } \leq \frac{1}{x} \int_{0}^{x}  e^{ - \frac{s^{2}}{2} } ds \leq 1$, so that 
\begin{align*}
I(x,a) = \frac{ \frac{1}{-a} \int_{x+a}^{x}  e^{ - \frac{s^{2}}{2} } ds }{ \frac{1}{x} \int_{0}^{x}  e^{ - \frac{s^{2}}{2} } ds } - 1 
\in \left[ e^{- \frac{(x+a)^2}{2} } - 1,  e^{\frac{x^2}{2}} -1  \right]. 
\end{align*}
Since $x +a <0$, we get
\begin{align*}
R(x,a) = \frac{1}{x+a} \left(
\frac{ \frac{1}{-a} \int_{x+a}^{x}  e^{ - \frac{s^{2}}{2} } ds }{ \frac{1}{x} \int_{0}^{x}  e^{ - \frac{s^{2}}{2} } ds }
- 1 \right)
\in \left[ \frac{1}{x+a} \left( e^{\frac{x^2}{2}} -1 \right),  \frac{1}{x+a} \left( e^{- \frac{(x+a)^2}{2} } - 1 \right) \right]. 
\end{align*}
As $1 \geq x > 0$ and $x+a < -1$, we have $\frac{1}{x+a} > -1$ and $e^{\frac{x^2}{2}} -1 \geq e^{1/2} - 1 > 1$.
Hence $\frac{1}{x+a} ( e^{\frac{x^2}{2}} -1 ) > -1$. 
Since we have shown in case (i) that $\frac{1}{t} ( e^{ - \frac{t^2}{2} } - 1 ) \geq -1$ for all $t \geq 0$,
we get $\frac{1}{x+a} ( e^{- \frac{(x+a)^2}{2} } - 1 ) \leq 1$. 
Therefore,  $|R(x, a)| \leq 1$ for all $1 \geq x > 0$ and $a <-2$. 
\end{proof}

\begin{lemma}\label{Lem-properties-of-V-L}
For any  $a\in \supp V$, there exist constants $c_1, c_2 >0$ such that, for any $n\geq 1$ and $x\in [a,\infty)$,
\begin{align} \label{lower bound of VL-002}
c_1 \leq V(x)L\left(  \frac{x}{\sigma \sqrt{n}} \right) \leq  c_2 \sqrt{n}.
\end{align}
\end{lemma}

\begin{proof}
For $x \in [a, \sigma \sqrt{n}]$, we have $V(x) \geq V(a) >0$ and $L(x) \geq \min \{ L(\frac{a}{\sigma \sqrt{n}}), L(1)\} \geq c'_1>0$ for any $n \geq 1$.
For $x > \sigma \sqrt{n}$, by \eqref{def-Lx-001} we have 
$
V(x) L(  \frac{x}{\sigma \sqrt{n}} ) = \frac{V(x)}{x} \sigma \sqrt{n} H(\frac{x}{\sigma \sqrt{n}})
\geq c \sigma \sqrt{n} H(1) > c''_1 >0,
$ 
for any $n \geq 1.$ 
Therefore the lower bound in \eqref{lower bound of VL-002} holds with $c_1=\max\{c'_1, c''_2 \}$.
For the upper bound, when $x\in [a,\sigma\sqrt{n})$, since $V(x)\leq c'_2\sqrt{n}$, 
we have $V(x) L(  \frac{x}{\sigma \sqrt{n}} )\leq c'_2 \sqrt{n } L(1)$. 
When $x\geq \sigma\sqrt{n}$,
since $\frac{V(x)}{x}\sim 1$, we have 
$V(x) L(  \frac{x}{\sigma \sqrt{n}} ) = \frac{V(x)}{x} \sigma \sqrt{n} H(\frac{x}{\sigma \sqrt{n}}) \leq c''_2 \sqrt{n}$.
Thus, the upper bound in \eqref{lower bound of VL-002} follows. 
\end{proof}

The following lemma will play a key role in the proof of Theorem \ref{Theor-probIntUN-002}.

\begin{lemma} \label{Holder prop for int ell} 
There exists a constant $c>0 $ such that, for any $x, x' \in \bb R$ with $|x-x'|\leq 1$, 
and any measurable set $D\subseteq \bb R$,
 \begin{align*} 
\left| \int_D \ell_h(x', y) dy  -  \int_D \ell_h(x, y) dy \right|  \leq c L(x) |x-x'|. 
\end{align*}
\end{lemma}

\begin{proof}
By the definition of $\ell_H$ and $\ell_h$ (cf.\ \eqref{def of func h-001} and \eqref{def of func ell_h-001}), we have $\ell_h(x,y) = L(x)\ell_H(x,y)$, 
which implies that, for any $x \in \bb R$ and $a \in [-1, 1]$, 
\begin{align} \label{bound integral ell-001}
& \int_D \left(\ell_h(x+a, y) -\ell_h(x, y) \right) dy \notag\\ 
& = L(x+a) \int_D  \ell_H(x+a, y)  dy - L(x) \int_D  \ell_H(x, y) dy  \notag\\
& = L(x+a)  \int_D  \left( \ell_H(x+a, y) - \ell_H(x, y) \right)  dy
 + \left( L(x+a) - L(x) \right) \int_D  \ell_H(x, y) dy\notag\\
&=: J_1 +J_2. 
\end{align}
For the first term,
there exists $\theta \in (0,1)$ such that 
\begin{align}\label{taylor for h-001}
\ell_H(x+a, y) - \ell_H(x, y) = a  \frac{d}{dx'} \ell_H(x', y) \Big|_{x' = x + \theta a}. 
\end{align}
By the definition of $\ell_H$ (cf.\ \eqref{def of func h-001} and \eqref{Def-Levydens}), we have 
\begin{align} \label{derivative of h-001}
\frac{d}{dx} \ell_H(x, y) 
 =  \frac{1}{H(x)^2} \Big[ H(x) \Big(  \phi'(x-y) - \phi'(x+y) \Big) 
- 2 \phi(x) \Big( \phi(x-y) - \phi(x+y) \Big) \Big].
\end{align}

First assume that $|x|>1$. Then, from \eqref{derivative of h-001}, it follows that
\begin{align*} 
\Big|\frac{d}{dx} \ell_H(x, y) \Big| 
& \leq  \frac{1}{|H(x)|} \Big(  |\phi'(x-y)| + |\phi'(x+y)|  \Big)  +  \frac{2\phi(x)}{H(x)^2} \Big(  |\phi(y+x)| +  |\phi(y-x)|  \Big) \notag \\ 
& \leq  c \Big(  |\phi'(x-y)| + |\phi'(x+y)|  +  |\phi(y+x)| +  |\phi(y-x)|  \Big). 
\end{align*}
Denote $g_0 (t) = \sup_{|z|\leq 1} \phi(t + z) $ and $g_1 (t) = \sup_{|z|\leq 1} \phi'(t + z) $
and remark that $\int_{\bb R} g_k(t)dt <\infty $ for $k=0,1.$
Then there exists $c>0$ such that for any $y \in \bb R$ and $a \in [-1, 1]$, 
\begin{align*} 
\Big|\frac{d}{dx'} \ell_H(x', y) \Big|_{x' = x + \theta a} \Big| 
& \leq   c \Big(  g_1(x-y) + g_1(x+y)  +  g_0(y+x)| +  g_0(y-x)  \Big). 
\end{align*}
Substituting this into \eqref{taylor for h-001} and using Lemma \ref{lem-inequality for L}, we get 
\begin{align*} 
|J_1|  & \leq c a L(x+a)  \int_D \Big(  g_1(x-y) + g_1(x+y)  +  g_0(y+x)| +  g_0(y-x)  \Big)  dy \\
&\leq c a (1+a)L(x)  \left(\int_0^{\infty}  g_1(y) dy + \int_0^{\infty} g_0(y)dy \right) \\
&\leq c a L(x).
\end{align*}

Secondly, assume that $|x|\leq 1$.
We rewrite \eqref{derivative of h-001} as follows:
\begin{align} \label{derivative of h-006}
\frac{d}{dx} \ell_H(x, y) 
& = \frac{1}{H(x)^2} \Big[ \left( x H(x) + 2 \phi(x) \right) \left(  \phi(y+x) -  \phi(y-x)  \right)  \notag\\
& \qquad\qquad\qquad + y  H(x) \left(  \phi(y+x) +  \phi(y-x)  \right)  \Big]. 
\end{align}

By Taylor's expansion, there exist $\theta_1, \theta_2, \theta_3 \in (0,1)$ such that 
\begin{align*}
\phi(y+x) & = \phi(y) + x \phi'(y) + \frac{x^2}{2} \phi''(y+\theta_1 x) = \phi(y) - xy \phi(y) + \frac{x^2}{2} \phi''(y+\theta_1 x) , \notag\\
\phi(y-x) & = \phi(y) - x \phi'(y) + \frac{x^2}{2} \phi''(y+\theta_2 x)= \phi(y) + xy \phi(y) + \frac{x^2}{2} \phi''(y+\theta_2 x)
\end{align*}
and
\begin{align*} 
H(x) & =  2 x \phi(x) +  x^2 \phi'(\theta_3 x). 
\end{align*}
Hence, 
\begin{align*}
\frac{d}{dx} \ell_H(x, y) 
& =  \frac{1}{H(x)^2} \Bigg[ \left( x H(x) + 2 \phi(x) \right)
 \left(  -2 xy \phi(y) + \frac{x^2}{2} \phi''(y+\theta_1 x) - \frac{x^2}{2} \phi''(y+\theta_2 x)  \right) 
 \notag\\
& \quad + y  \left( 2 x \phi(x) +  x^2 \phi'(\theta_3 x)\right) \left( 2 \phi(y) 
  +  \frac{x^2}{2} \phi''(y+\theta_1 x) + \frac{x^2}{2} \phi''(y+\theta_2 x) \right) \Bigg] \\
& =  \frac{1}{H(x)^2} \Bigg[ \left( x H(x) + 2 \phi(x) \right)
 \left(  \frac{x^2}{2} \phi''(y+\theta_1 x) - \frac{x^2}{2} \phi''(y+\theta_2 x)  \right) 
 \notag\\
& \quad + y  \left( 2 x \phi(x) +  x^2 \phi'(\theta_3 x)\right) 
\left(  \frac{x^2}{2} \phi''(y+\theta_1 x) + \frac{x^2}{2} \phi''(y+\theta_2 x) \right) \Bigg].
\end{align*}
Denote $g_2 (t) = \sup_{|z|\leq 1} \phi''(t + z)$ so that $\int_{\bb R} g_2(t)dt <\infty $. 
Then, taking into account that $xL(x)=H(x)$, we obtain 
\begin{align*} 
\left| \frac{d}{dx} \ell_H(x, y) \right| 
& \leq  \frac{1}{H(x)^2} \Bigg[ \left( x^2 L(x) + 2 \phi(x) \right) x^2 g_2(y) + 2 x^2 |y| \phi(y) |H(x)|  
 \notag\\
& \quad + |y|  \left( 2 |x| \phi(x) +  x^2 g_1(0)  \right)  x^2 g_2(y)  + 2 |y| x^2 \phi(y) g_1(0) \Bigg] \notag \\
&=   \frac{x^2}{H(x)^2} \left( x^2 L(x) + 2 \phi(x) \right)  g_2(y) + 2  |y| \phi(y) |H(x)| \frac{x^2}{H(x)^2} 
 \notag\\
& \quad +  \left( 2 |x| \phi(x) +  x^2 g_1(0)  \right) \frac{x^2}{H(x)^2} |y|   g_2(y)  + 2 |y| \phi(y) g_1(0) \notag \\
&\leq   c \Big(  g_2(y) +  |y| \phi(y) + |y| g_2(y) \Big).
\end{align*}
Using \eqref{taylor for h-001} and the last bound, we have
\begin{align*} 
\int _{D}  \left|  \ell_H(x+a, y) - \ell_H(x, y)\right| dy \leq   c |a|  \int _{\bb R} \left[  g_2(y) +  |y| \phi(y) + |y| g_2(y) \right] dy
\leq c |a|. 
\end{align*}
Therefore, for the first term in the right-hand side of \eqref{bound integral ell-001} we have 
\begin{align*} 
|J_1| \leq c |a|  L(x+a) \leq c |a| L(x) (1+|a|) \leq c |a| L(x). 
\end{align*}
For the second term in in the right-hand side of \eqref{bound integral ell-001}, by Lemma \ref{lem-inequality for L}, 
\begin{align*} 
|J_2| \leq c |a| L(x) \int_{\bb R}  | \ell_H(x, y)| dy \leq 2 c |a| L(x). 
\end{align*}
This concludes the proof of the lemma.
\end{proof}

\subsection{Technical results}

The purpose of this subsection is to state several auxiliary results 
that will be used to establish the conditioned integral limit theorems 
(Theorems \ref{Theor-probtauUN-001} and \ref{Theor-probIntUN-002}). 
Throughout this section, unless otherwise specified, we assume the following conditions: 
$\bb E (X_1) = 0$, $\bb E (X^2_1)= \sigma^2  > 0$ and there exists a constant $\delta > 0$ 
such that $\bb E (|X_1|^{2+\delta})  < \infty.$

The following functional central limit theorem, proved by Sakhanenko \cite{Sak06} (see also \cite{KMT-2}), 
provides a way to couple out the random walk with the Brownian motion.

\begin{lemma}\label{FCLT}
There exists a construction of the random walk $(S_n)_{n\geq 0}$ on the initial probability space
together with a continuous time Brownian motion $(B_t)_{t\geq 0}$ such that for any $s>0$ and $n\geq 1$,
\begin{align*}
\bb{P} \left( \sup_{0\leq t\leq 1} \left| S_{[nt]} - \sigma B_{nt} \right| > \sigma n^{1/2 - s}  \right)
 \leq  \frac{c_{\delta,s}}{ n^{ \frac{\delta}{2} - (2+\delta) s} }, 
\end{align*}
where $c_{\delta,s}$ is a constant depending on $\delta$ and $s$. 
\end{lemma}
Note that the assertion of the lemma becomes effective when $s \in (0,  \frac{\delta}{2(2+\delta)})$.

The following lemma is elementary. 

\begin{lemma}\label{Lem-Burkholder}
For any $r \in (0, 2 + \delta]$, there exists a constant $c>0$ such that for any $n \geq 1$, 
\begin{align*}
\bb E \left( \max_{1 \leq k \leq n} |S_k|^r \right) \leq c n^{r/2}. 
\end{align*}
\end{lemma}

\begin{proof}
By Rosenthal's inequality, for any $s \in [2, 2 + \delta]$, there exists a constant $c>0$ such that for any $n \geq 1$, 
\begin{align*}
\bb E \left( \max_{1 \leq k \leq n} |S_k|^s \right) 
\leq  c \sum_{i = 1}^n  \bb E \left( |X_i|^s \right)  +   c \left[ \sum_{i = 1}^n  \bb E \left( |X_i|^2 \right) \right]^{s/2}
\leq  c n^{s/2}. 
\end{align*}
Using this with $s=2$ and H\"older's inequality, 
we get that, for $r \in (0, 2 + \delta]$, there exists a constant $c>0$ such that for any $n \geq 1$, 
\begin{align*}
\bb E \left( \max_{1 \leq k \leq n} |S_k|^r \right) 
\leq  \left[ \bb E \left( \max_{1 \leq k \leq n} |S_k|^2 \right)  \right]^{r/2}
\leq c n^{r/2}, 
\end{align*}
completing the proof of the lemma.  
\end{proof}

We shall need the following Fuk-Nagaev inequality (cf.\ \cite{FN71}).

\begin{lemma}  \label{FNB-inequality}
Assume that $\bb EX_i =0$ and $\bb E (X_i^2) = \sigma^2 >0$.  
Then, for any $n \geq 1$ and $x, y >0$, 
\begin{align*}
\bb P \left( |S_n| > x,  \max_{1 \leq i \leq n} |X_i| \leq y \right)
\leq 2 e^{x/y } \left( \frac{n \sigma^2}{xy} \right)^{ x/ y }. 
\end{align*}
\end{lemma}

The following lemma, adapted from \cite[Lemma 5.8]{GQX24}, is a consequence of the central limit theorem. 

\begin{lemma} \label{lemma-sup  S_k -001} 
There exists a constant $c > 0$ such that, 
for any $M \geq 1$ and $n\geq 1$, 
\begin{align*} 
 \sup_{x \in \bb R}
\bb P\left( \sup_{1\leq k \leq n}  | x + S_k | \leq M   \right) \leq c e^{-  \frac{n}{2 M^2}}.
\end{align*}
 \end{lemma}

\begin{proof}
Let $m = [\rho^{-2} M^2]$ and $K=\left[ n/ m \right]$, where $\rho > 0$ will be chosen later.
It is easy to see that, for any $t \in \bb R$,
\begin{align}
\mathbb{P}\left( \max_{1\leq k\leq n}\left\vert t + S_{k}\right\vert \leq  M \right)   
\leq \mathbb{P}\left( \max_{1\leq k\leq K}\left\vert t+S_{km}\right\vert \leq M \right) .  \label{nu000}
\end{align}%
Using the independence of the random variables $(X_k)$, it follows that 
\begin{align*}
\mathbb{P}\left( \max_{1\leq k\leq K}\left\vert x +S_{km}\right\vert \leq M \right)  
\leq \mathbb{P}\left( \max_{1\leq k\leq K-1}\left\vert x +S_{km}\right\vert \leq M \right) 
 \sup_{x'\in \mathbb{R}} \mathbb{P} \left( \left\vert x'+S_{m}\right\vert \leq M \right) ,
\end{align*}
from which iterating, we get
\begin{align} \label{Piterations-001}
\mathbb{P} \left( \max_{1\leq k\leq K}\left\vert t +S_{km}\right\vert \leq  M \right) 
 \leq \left(  \sup_{t'\in \mathbb{R}} \mathbb{P} \left( \left\vert x'+S_{m}\right\vert \leq M \right) \right) ^{K}.
\end{align}
By the central limit theorem, 
there exist a sequence $r_m\to 0$ as $m\to\infty$ and a constant $q\in (0,1)$ such that, for any $x' \in \bb R$ and $m \geq 1$, 
\begin{align*}
\mathbb{P}\left( \left|  x' +S_{m} \right| \leq M \right) 
 = \mathbb{P} \left( \frac{S_{m}}{\sqrt{m}} \in \left[ \frac{-M - x' }{\sqrt{m}},  \frac{M- t'}{\sqrt{m}} \right] \right)  
 \leq  \int_{  \frac{-M-x'}{\sqrt{m}} }^{ \frac{M-x'}{\sqrt{m} } }  \phi_{\sigma^2} (u) du  +  r_m
\leq  q, 
\end{align*}
where for the last inequality we take $\rho >0$ sufficiently small and $m$ sufficiently large so that 
$\frac{2M}{\sqrt{m}} + r_m \leq 4 \rho + r_m  \leq q $. 
The assertion of the lemma follows from \eqref{nu000} and \eqref{Piterations-001}. 
\end{proof}

The following bound is due to Rogozin (\cite[Corollary 4]{Rog77}).

\begin{lemma}\label{Lem-bound for V-001}
There exists $c>0$ such that, for any $x \geq 1$, 
\begin{align}\label{inequa-Rogozin}
0\leq  \frac{V(x)}{x} - 1  \leq c x^{-\delta}. 
\end{align}
\end{lemma}

We continue with the following preliminary bound for the expectation of 
the random walk $(x + S_n)_{n \geq 1}$ killed at the exit time $\tau_x$.

\begin{lemma}\label{Lem-MKillTn}
There exists a constant $c > 0$ such that for any  $n \geq 1$ and $x \in \bb R$, 
\begin{align}\label{weak  bound for V_n-001}
\max\{x, 0\}  \leq  \mathbb{E} \left(x + S_n; \tau_x > n \right) 
\leq V(x) = -\mathbb{E} (S_{\tau_x}). 
\end{align}
\end{lemma}

\begin{proof}
By the optional stopping theorem, for any $x\in \bb R$, we get
\begin{align}\label{Pf_Optional_thm}
\mathbb{E} (x + S_n; \tau_x > n ) 
& = \bb E (x + S_{n})  - \mathbb{E} (x + S_{n}; 1\leq \tau_x \leq n )\notag \\
&= x  - \mathbb{E} ( x + S_{\tau_x}; 1\leq \tau_x \leq n ).
\end{align}
Since, on the event $\{ 1\leq \tau_x \leq n \}$ it holds  $- (x + S_{\tau_x}) \geq 0$, we get 
\begin{align*}
0 \leq - \mathbb{E} ( x + S_{\tau_x}; 1\leq \tau_x \leq n ) \leq - \mathbb{E} ( x + S_{\tau_x} ) 
= - x - \mathbb{E} ( S_{\tau_x} ) = - x + V(x), 
\end{align*}
where the last equality holds due to \eqref{Equ-Vx}. 
Substituting this into \eqref{Pf_Optional_thm} completes the proof of the lemma. 
\end{proof}

 We will need the following elementary inequality from \cite[Lemma 5.10]{GQX24}. 
 
 \begin{lemma}\label{Lem-trick}
For any $x \in \bb R$ and any random variable $Z$,
\begin{align*} 
\bb E^{1/2}\left( \left(x + Z\right)^2; x+Z\geq 0 \right) \leq \max\{x,0 \} + \bb E^{1/2} (Z ^2).
\end{align*}
\end{lemma}
\begin{proof}
If $x \leq 0$, we have $Z\geq -x\geq 0 $ since $x+Z\geq 0$. Hence $(x+Z)^2 \leq Z^2$ and the assertion follows.
If $x >0$, using $\bb EZ \leq \bb E^{1/2} (Z ^2)$, we get 
$x^2 + 2x\bb EZ + \bb E Z^2  \leq x^2 +2 x \bb E^{1/2} (Z ^2) + \bb E Z^2,$ 
from which the assertion follows. 
\end{proof}

\section{Proofs for the persistence probability} \label{SecProof Theor-probtau-001}

In this section, we give a proof of Theorem \ref{Theor-probtauUN-001}  
by using the bounds shown in Section  \ref{sec-properties heat kern}
and the  functional central limit theorem (Lemma \ref{FCLT}). 
All over this section we assume that  $\bb E (X_1) = 0$, $\bb E (X^2_1)= \sigma^2  > 0$  and that there exists $\delta > 0$ 
such that $\bb E (|X_1|^{2+\delta})  < \infty.$

\subsection{Auxiliary results}

Let $\ee \in [0, \frac{1}{2}]$. 
For any $x \in \bb R$ and $n\geq 1$, consider the first time 
when the random variable $|x + S_k| $ exceeds the level $M_{n} := n^{1/2 - \ee}$:
\begin{align}\label{Def_nun}
\nu_{n}  = \nu_{n,\ee}  = \inf \left\{ k \geq 1: |x + S_k| > M_{n} \right\}. 
\end{align}
The following lemma gives the tail behavior of $\nu_{n}$.  Let $N_n = N_n^{\ee} :=[\gamma n^{1-2\ee} \log n]$, 
where $\gamma>0$ is chosen to be a sufficiently large constant.

\begin{lemma} \label{Lemma 2-0} 
For any $\gamma>0$, there exists a constant $c >0$ 
such that for any $n\geq 1$, $x\in \bb R$ and $\ee \in [0, \frac{1}{2}]$, 
\begin{align*} 
\mathbb{P} \left( \nu_{n} > N_n \right)= \mathbb{P} \left( \nu_{n} > [\gamma n^{1-2\ee} \log n] \right)
 \leq c n^{-\gamma/2}. 
\end{align*}
\end{lemma}

\begin{proof}
It suffices to consider the case $n\geq 2$. 
Since for $\ee\in [0,1/2]$, we have $M_n= n^{1/2-\ee}\geq 1$.  
Using Lemma \ref{lemma-sup  S_k -001}, we get
\begin{align*} 
\bb P\left(\nu_n \geq N_n \right) 
&=  \mathbb{P} \left( \sup_{ 1\leq k \leq N_n } |x + S_k | \leq M_{n} \right) \\
&\leq c \exp \left(- \frac{ N_n }{2M_n^2} \right) 
 = c \exp \left(- \frac{ [\gamma n^{1-2\ee} \log n] }{2n^{1-2\ee}} \right)  \leq c n^{-\gamma/2}, 
\end{align*}
which proves the assertion of the lemma. 
\end{proof}

For any $x\in \bb R$ and $n\geq 1$, set 
\begin{align*}
E_{x, n, \ee} = \mathbb{E} \left( x + S_{\nu_{n,\ee}}; \tau_x > \nu_{n,\ee}, \nu_{n,\ee} \leq N_n^{\ee} \right). 
\end{align*}
The following lemma will be used repeatedly in the subsequent analysis to establish
 Theorems \ref{Theor-probtauUN-001}  and \ref{Theor-probIntUN-002}. 
It demonstrates that $E_{x, n, \ee}$ is well approximated by the harmonic function $V$.

\begin{lemma}\label{lemma E1 and E2}
For any $\ee \in  [0, \frac{1}{2}]$ and $\gamma > 0$, 
there exists a constant $c>0$ such that for any $n \geq 1$ and $x \in \bb R$, 
\begin{align}\label{control of E1-002}
E_{x, n, \ee} \leq  V(x)  \leq  \left( 1 + n^{-(1/2 - \ee) \delta} \right) E_{x, n, \ee}
 + c \frac{1 +V(x)}{n^{\gamma/8}}. 
\end{align}
\end{lemma}

\begin{proof}
Let $n\geq 1$,  $x\in \bb R$ and $\ee\in [0,\frac{1}{2}]$. 
By \eqref{Doob transf}, the sequence $\big( V(x + S_n) \mathds{1}_{\{ \tau_x > n \}} \big)_{n \geq 1}$ is a martingale
with respect to the filtration $\mathscr F_{n} = \sigma (X_1, \ldots, X_n)$. 
Therefore, by the optional stopping theorem, 
\begin{align} \label{start split V-001}
V(x) & =  \mathbb{E} \left( V(x + S_n); \tau_x > n \right) \notag\\
& =  \mathbb{E} \left( V(x + S_n); \tau_x > n, \nu_n \leq N_n \right) 
     + \mathbb{E} \left( V(x + S_n); \tau_x > n, \nu_n > N_n  \right) \nonumber\\ 
  & =  \mathbb{E} \left( V(x + S_{\nu_n}); \tau_x > \nu_n, \nu_n \leq N_n \right)  
   + \mathbb{E} \left( V(x + S_n); \tau_x > n, \nu_n > N_n \right).   
\end{align}
By Lemma \ref{Lem-MKillTn}, 
we have $V(x) \geq \max\{x, 0\}$ for $x\in \bb R$, which implies that
the last term in the right-hand side of \eqref{start split V-001} is non-negative.   
Therefore, from \eqref{start split V-001} and the fact that $V(x + S_{\nu_n}) \geq x + S_{\nu_n}$
on the event $\{\tau_x > \nu_n\}$, we get that, for any $x\in \bb R$,
\begin{align*}
V(x)  \geq \mathbb{E} \left( V(x + S_{\nu_n}); 
     \tau_x > \nu_n, \nu_n \leq N_n \right) \geq 
     \mathbb{E} \left( x + S_{\nu_n}; 
     \tau_x > \nu_n, \nu_n \leq N_n \right) = E_{x, n, \ee}, 
\end{align*}
which proves the left-hand side bound of \eqref{control of E1-002}.
For the right-hand side bound, 
since $x'=x + S_{\nu_n}\geq n^{1/2 - \ee}$ on the event $\{ \tau_x > \nu_n  \}$, using Lemma \ref{Lem-bound for V-001}, 
we get
\begin{align*}
V(x + S_{\nu_n}) \leq  \left( 1 + c (x + S_{\nu_n})^{-\delta} \right) (x + S_{\nu_n}) \leq \left( 1 + n^{-(1/2 - \ee) \delta} \right) (x + S_{\nu_n}). 
\end{align*}
It follows that 
\begin{align}\label{control of E2 0}
\mathbb{E} \left( V(x + S_{\nu_n}); \tau_x > \nu_n, \nu_n \leq N_n \right)
& \leq \left( 1 + n^{-(1/2 - \ee) \delta} \right) \mathbb{E} \left(  x + S_{\nu_n}; 
  \tau_x > \nu_n, \nu_n \leq N_n \right) \notag\\
 & = \left( 1 + n^{-(1/2 - \ee) \delta} \right) E_{x, n, \ee}. 
\end{align}
Combining \eqref{start split V-001} and \eqref{control of E2 0}, we have, for any $n\geq 1$ and $x\in \bb R$, 
\begin{align}\label{control of E21 1}
V(x) 
 \leq  \left( 1 + n^{-(1/2 - \ee) \delta} \right) E_{x, n, \ee}
 +  \mathbb{E} \left( V(x + S_n); \tau_x >n, \nu_n > N_n \right).
\end{align}
By the Cauchy-Schwarz inequality, the bound $V(x') \leq c(1 + x')$ for $x'\geq 0$ (cf.\ Lemma \ref{Lem-bound for V-001}) 
and Lemma \ref{Lemma 2-0}, 
we get that for any $\ee \in [0,1/2]$ and $\gamma>0$, there exists a constant $c>0$ such that, for any $n \geq 1$ and $x\in \bb R$,  
\begin{align}\label{control of E22 all}
 \mathbb{E} \left( V(x + S_n); \tau_x >n, \nu_n > N_n \right)  
& \leq c (V(x) + n^{1/2}) \mathbb{P}^{1/2} \left( \nu_n > N_n \right)  \nonumber\\
& \leq   c (1 + V(x))  n^{-\gamma/8}. 
\end{align}
Combining \eref{control of E21 1} and \eref{control of E22 all}, 
we obtain that for any $\ee \in [0,1/2]$ and $\gamma>0$, 
there exists a constant $c>0$ such that, for any $n \geq 1$ and $x\in \bb R$,
\begin{align*}
V(x) \leq  \left( 1 + n^{-(1/2 - \ee) \delta} \right) E_{x, n, \ee} + c (1 +V(x))  n^{-\gamma/8}. 
\end{align*}
This finishes the proof of \eqref{control of E1-002}.
\end{proof}

Using Lemmas \ref{Lemma 2-0} and \ref{lemma E1 and E2}, 
we show the following upper bound for the persistence probability $\bb P \left(\tau_x >n \right)$,
which will be applied in the proof of Lemma \ref{Bound for J_4}.

\begin{lemma} \label{new bound for Ptau_x>n-001}
There exists a constant $c >0$ such that 
for any $n \geq 1$ and $x \in \bb R$, 
\begin{align*}
\bb P \left(\tau_x >n \right) \leq c  \frac{1+V(x)  L\left(\frac{x}{\sigma\sqrt{n}}\right) }{n^{1/2}}. 
\end{align*}
\end{lemma}
\begin{proof} 
Using Lemma \ref{Lemma 2-0} with $\ee=0$ and the fact that $x + S_{\nu_{n,0}} \geq n^{1/2}$
on the event $\{ \tau_x > \nu_{n,0} \}$, 
we get that, for any $\gamma>0$, there exists a constant $c>0$ such that
for any $n \geq 1$ and $x \in \bb R$, 
\begin{align*}
\mathbb{P} ( \tau_x > n ) 
& =  \mathbb{P} \left( \tau_x > n, \nu_{n,0} \leq N_n^{0} \right)
  + \mathbb{P} \left( \tau_x > n, \nu_{n,0} > N_n^{0} \right)  \nonumber\\
& \leq  \frac{1}{n^{1/2}}
  \mathbb{E} \left(x + S_{\nu_{n,0}}; \tau_x > \nu_{n,0}, \nu_{n,0} \leq N_n^{0} \right)
 + c n^{-\gamma/2}. 
\end{align*}
Using Lemma \ref{lemma E1 and E2}, we have
\begin{align*} 
E_{x,n,0} = \mathbb{E} \left(x + S_{\nu_{n,0}}; \tau_x > \nu_{n,0}, \nu_{n,0} \leq N_n^{0} \right) \leq V(x). 
\end{align*}
Therefore, choosing $\gamma >0$ large enough, there exists a constant $c$ such that, 
for any $n\geq 1$ and $x \in \bb R$,
\begin{align*}
\bb P \left(\tau_x >n \right) \leq c \frac{1+V(x) }{n^{1/2}}. 
\end{align*}
From Lemma \ref{lem-inequality for L}, there exists $c>0$ such that $L\left(\frac{x}{\sigma\sqrt{n}}\right) \geq c L(0)$ 
for any $n\geq 1$ and $0\leq x\leq \frac{\sigma}{2} n^{1/2}$. 
Therefore, for any $0\leq x\leq \frac{\sigma}{2} n^{1/2}$,
\begin{align}\label{Inequa-tau-x-001}
\bb P \left(\tau_x >n \right)  \leq c \frac{1+V(x) }{ n^{1/2}}
\leq c \frac{1+V(x)  L\left(\frac{x}{\sigma\sqrt{n}}\right) }{n^{1/2}}. 
\end{align}
For $x > \frac{\sigma}{2} n^{1/2}$, we have, with some constant $c>0,$ 
\begin{align*} 
\bb P(\tau_x >n ) \leq c H\left(\frac{x}{\sigma \sqrt{n}}\right) = c \frac{x}{  \sqrt{n}} L\left(\frac{x}{\sigma \sqrt{n}}\right) 
\leq c \frac{V(x)  L\left(\frac{x}{\sigma \sqrt{n}}\right)}{  \sqrt{n}}. 
\end{align*}
For $x<0$, 
since $\bb P \left(\tau_x >n \right)  
\leq \bb P \left(\tau_0 >n \right)$, using \eqref{Inequa-tau-x-001}, it holds that 
\begin{align}\label{proba-tau-x-001}
\bb P \left(\tau_x >n \right)  
\leq \bb P \left(\tau_0 >n \right) \leq   \frac{c_1 }{ n^{1/2}}
\leq c_1 \frac{1+ V(x)  L\left(\frac{x}{\sigma \sqrt{n}}\right) }{n^{1/2}}. 
\end{align}
This ends the proof of the lemma.
\end{proof}

As a consequence of Lemmas \ref{lem-inequality for L} and \ref{FCLT}, 
the next result provides relative error estimates for the probability $\bb P(\tau_x >n)$ with the Dirichlet heat kernel $H$
defined by \eqref{Heat func-001}. 

\begin{lemma} \label{Lemma 2-KMTrate}  \ \\ 
1. There exists a constant $c >0$ such that, for any $n\geq 1$, $a \in (0,1]$ and $x \geq a\sigma \sqrt{n}$, 
\begin{align}\label{Relative-Heat-001}
\left| \frac{\mathbb{P} \left( \tau_x > n \right)}{H\left(\frac{x}{\sigma\sqrt{n}}\right) } - 1 \right|  
\leq \frac{c}{a} n^{-\frac{\delta}{2(3+\delta)}}. 
\end{align}
2. Moreover, if $\delta \geq 1$, then there exists a constant $c >0$ such that, for any $n\geq 1$, $a \in (0,1]$ and $x \geq a\sigma \sqrt{n}$, 
\begin{align}\label{Relative-Heat-002}
\left| \frac{\mathbb{P} \left( \tau_x > n \right)}{H\left(\frac{x}{\sigma\sqrt{n}}\right) } - 1 \right|  
\leq \frac{c}{a\sqrt{n}} .
\end{align}
\end{lemma}

\begin{proof}
By  Lemma \ref{FCLT}, for any $s>0$, there exists $c = c_{s, \delta}>0$ such that, for any $n\geq 1$ and $x \in \bb R$,
\begin{align} \label{BEE-001}
H\left(\frac{x}{\sigma\sqrt{n}} - n^{-s}\right)  - \frac{c}{n^{\frac{\delta}{2} - s(2+\delta) }}
\leq \bb P(\tau_x >n) 
\leq H\left(\frac{x}{\sigma\sqrt{n}} + n^{-s}\right)  + \frac{c}{n^{\frac{\delta}{2} - s(2+\delta) }}.
\end{align}
Using Lemma \ref{lem-inequality for L},  we have
\begin{align*} 
H\left(\frac{x}{\sigma\sqrt{n}} + n^{-s}\right) 
&= \left(\frac{x}{\sigma\sqrt{n}} + n^{-s}\right) L\left(\frac{x}{\sigma\sqrt{n}} + n^{-s}\right) \\
&\leq \left(\frac{x}{\sigma\sqrt{n}} + n^{-s}\right) L\left(\frac{x}{\sigma\sqrt{n}}\right) \left(1 +  c n^{-s}\right) \\
& \leq  H\left(\frac{x}{\sigma\sqrt{n}}\right) \left(1 +  c n^{-s}\right) 
+  L\left(\frac{x}{\sigma\sqrt{n}}\right) n^{-s} \left(1 +  cn^{-s}\right).
\end{align*}
Taking into account that $x\geq a\sigma\sqrt{n}$, we get
\begin{align} \label{BEE-002}
H\left(\frac{x}{\sigma\sqrt{n}} + n^{-s}\right) 
&\leq  H\left(\frac{x}{\sigma\sqrt{n}}\right) \left(1 +  cn^{-s}\right) 
+ c H\left(\frac{x}{\sigma\sqrt{n}}\right) \frac{1}{an^{s}} \notag\\
& \leq   H\left(\frac{x}{\sigma \sqrt{n}}\right) \left(1 + c' \frac{1}{an^{s}} \right) .
\end{align}
Since $H$ is increasing on $\bb R_+$, there exists $c>0$ such that for any $a \in (0,1]$ and $x\geq a\sigma\sqrt{n}$, we have 
\begin{align}\label{lower-H-func}
H\left(\frac{x}{\sigma\sqrt{n}}\right) \geq H(a)\geq c a. 
\end{align}
From \eqref{BEE-001} and \eqref{BEE-002}, by choosing $s=\frac{\delta}{2(3+\delta)}$, we get the following upper bound: 
\begin{align*}
\frac{\mathbb{P} \left( \tau_x > n \right)}{H\left(\frac{x}{\sigma\sqrt{n}}\right) } - 1  
\leq \frac{c}{a} n^{-\frac{\delta}{2(3+\delta)}}. 
\end{align*}
The lower bound can be obtained in the same way, so that \eqref{Relative-Heat-001} holds.

Now we show the second assertion \eqref{Relative-Heat-002}. 
As a consequence of Nagaev's bound \eqref{Nag-bound-001}, 
there exists $c>0$ such that for any $n \geq 1$ and $x\geq 0$, 
\begin{align*}
\left| \frac{\mathbb{P} \left( \tau_x > n \right)}{H\left(\frac{x}{\sigma\sqrt{n}}\right) } - 1 \right|  
\leq \frac{c}{\sqrt{n} H\left(\frac{x}{\sigma\sqrt{n}}\right)}.
\end{align*}
Since $x\geq a\sigma\sqrt{n}$, the bound \eqref{Relative-Heat-002} follows from \eqref{lower-H-func}. 
\end{proof}

\subsection{Proof of Theorem \ref{Theor-probtauUN-001}}
It suffices to consider the case $n\geq 2$. 
Recall that $M_n = n^{1/2-\ee}$ and $N_n=[\gamma n^{1-2\ee} \log n]$, where $\ee \in (0, \frac{1}{2})$
and $\gamma>0$ will be chosen to be a sufficiently large constant. 
Using the Markov property, we have for any  $n\geq 1$ and $x\in \bb R$,  
\begin{align}\label{start bndtau-001}
\mathbb{P} \left( \tau_x > n \right)  = J_1 + J_2, 
\end{align}
where 
\begin{align*}
& J_1 = \mathbb{P} \left( \tau_x > n, \nu_{n} > N_n \right),  \notag\\
& J_2 = \mathbb{P} \left( \tau_x > n, \nu_{n} \leq N_n \right)  
 =  \sum_{k = 1}^{N_n}
  \int_{M_{n}}^\infty \mathbb{P}(\tau_{x'}>n-k)  
\mathbb{P} \left( x + S_k\in dx', \tau_x > k, \nu_{n} = k \right). 
\end{align*}

\textit{Bound of $J_1$}. 
By Lemma \ref{Lemma 2-0}, for any $\gamma>0$, there exists a constant $c >0$ 
such that for any $n\geq 1$, $x\in \bb R$ and $\ee \in [0, \frac{1}{2}]$, 
\begin{align} \label{upper bound-J1-001}
J_1 \leq  \mathbb{P} \left( \nu_{n} > N_n \right) 
\leq   c n^{-\gamma/2}. 
\end{align}

\textit{Upper bound of $J_{2}$}. 
Let $k \in [1, N_n]$ and $x' \geq M_n = n^{1/2-\ee}$. 
Using Lemma \ref{Lemma 2-KMTrate} with $a=n^{-\ee}$,  we get
\begin{align} \label{Probab-tay n-k-001}
  \mathbb{P}( \tau_{x'} > n-k)  
  \leq   H \left( \frac{x'}{\sigma\sqrt{n-k}} \right) \left(  1 + c n^{\ee} r_{n,\delta}  \right),
\end{align}
where $r_{n,\delta}= n^{-\frac{\delta}{2(3+\delta)}}$ is the remainder term in \eqref{Relative-Heat-001}. 
Substituting \eqref{Probab-tay n-k-001} into $J_2$ gives 
\begin{align} \label{bound J_3-001a}
J_{2} \leq  \left(  1 + c n^{\ee} r_{n,\delta}  \right)  
 \sum_{k = 1}^{N_n} \bb E \left(    H \left( \frac{x+S_k}{\sigma\sqrt{n-k}}\right); \tau_x > k, \nu_n = k  \right). 
\end{align}
Since for any $k \in [1, N_n]$, 
\begin{align} \label{very aux bound sqrt n-001}
\frac{\sqrt{n}}{\sqrt{n-k}}  
= 1+ \frac{k}{\sqrt{n-k}(\sqrt{n}+\sqrt{n-k})} \leq 1+ c\frac{k}{n} \leq 1+ c \frac{N_n}{n},  
\end{align}
by Lemma \ref{lem-inequality for H}, 
we have that for any $x' \geq 0$, 
\begin{align*} 
H \left( \frac{x'}{\sigma\sqrt{n-k}}\right) \leq H \left( \frac{x'}{\sigma\sqrt{n}}\right) \left(1+c\frac{N_n}{n}\right). 
\end{align*}
Combining this with \eqref{bound J_3-001a}, we obtain 
 \begin{align} \label{bound J_3-001c}
J_{2} 
& \leq 
 \left(  1 + c n^{\ee} r_{n,\delta}  \right) \left(1+ c\frac{N_n}{n}\right) 
 \sum_{k = 1}^{N_n} \bb E \left(    H \left( \frac{x+S_k}{\sigma\sqrt{n}}\right); \tau_x > k, \nu_n = k  \right)     \notag \\
  & \leq  \left(  1 + c n^{\ee} r_{n,\delta} + c\frac{N_n}{n}  \right) 
  \sum_{k = 1}^{N_n} \bb E \left(    H \left( \frac{x+S_k}{\sigma\sqrt{n}}\right); \tau_x > k, \nu_n = k  \right)     \notag \\
 & =  \left(  1 + c n^{\ee} r_{n,\delta} + c\frac{N_n}{n}  \right)   (J_{3} + J_{4}),
\end{align}
where 
\begin{align} \label{term-J4-001}
J_{3} 
= \sum_{k = 1}^{N_n} \bb E
      \left(H \left( \frac{x+S_k }{\sigma\sqrt{n}}  \right);   
      \frac{|S_k|}{\sigma\sqrt{n}}  > \alpha n^{-\ee}\log n,  \tau_x > k, \nu_n = k \right)
\end{align}
and
\begin{align} \label{term-J5-001}
J_{4} 
= \sum_{k = 1}^{N_n} \bb E
      \left(H \left( \frac{x+S_k }{\sigma\sqrt{n}}  \right); 
      \frac{|S_k|}{\sigma\sqrt{n}}  \leq \alpha n^{-\ee}\log n,  \tau_x > k, \nu_n = k \right), 
\end{align}
with the constant $\alpha>0$ whose value will be chosen below. 

The remainder of this section focuses on deriving upper and lower bounds for $J_3$ and $J_4$. 
Specifically, we will demonstrate that $J_3$ is negligible, whereas $J_4$ constitutes the dominant term.

\begin{lemma} \label{Bound for J_4}
For any  $\ee \in (0,\frac{1}{2}]$, there exists a constant $c=c_{\ee}>0$ 
such that, for any $n\geq 2$ and $x\in \bb R$,
\begin{align*}
J_{3}  \leq c  \frac{1 + V(x) L\left(\frac{x}{\sigma\sqrt{n}}\right)}{ n^{\frac{1}{2} + \frac{\delta}{2}  - \delta\ee  } } (\log n)^{3/2}.   
\end{align*}
\end{lemma}

\begin{proof}
We will use  the Fuk-Nagaev inequality (Lemma \ref{FNB-inequality}).  Let
\begin{align} \label{def eta_n-001}
\eta_n=\inf \left\{ k\geq 1: \frac{|X_k|}{\sigma \sqrt{n}}   \geq \alpha n^{-\ee} \right\}, 
\end{align} 
where $\alpha>0$ is given in \eqref{term-J4-001}. 
By \eqref{term-J4-001}, we have that for $n\geq 2$,
\begin{align} \label{bound J_4001}
J_{3} 
&= \sum_{k = 1}^{N_n} \bb E
      \left(H \left( \frac{x+S_k}{\sigma\sqrt{n}}  \right);   \frac{|S_k|}{\sigma\sqrt{n}}  > \alpha n^{- \ee}\log n,  \tau_x > k,
      \eta_n > k,  \nu_n = k \right) \notag\\
&\quad + \sum_{k = 1}^{N_n} \bb E
      \left(H \left( \frac{x+S_k}{\sigma\sqrt{n}}  \right);   \frac{|S_k|}{\sigma\sqrt{n}}  > \alpha n^{- \ee} \log n,  \tau_x > k,
      \eta_n \leq k,  \nu_n = k \right) \notag\\
&=: J_{31} + J_{32}.
\end{align}
For the first term $J_{31}$, since $H\leq 1$,  we have 
\begin{align} \label{bound J_41-001}
J_{31}  \leq  \sum_{k = 1}^{N_n}  \bb P \left( \frac{|S_{k}|}{\sigma\sqrt{n}} > \alpha n^{-\ee}\log n,  \eta_n > k \right). 
\end{align}
Applying the Fuk-Nagaev inequality (Lemma \ref{FNB-inequality}) with $x = \alpha \sigma n^{1/2 - \ee}\log n$ 
and $y = \alpha \sigma n^{1/2-\ee}$, we get that for any $1 \leq k \leq N_n$, 
\begin{align}\label{Apply-Fuk-Nagaev-aa}
&\bb P \left( \frac{|S_k|}{\sqrt{n}}  > \alpha \sigma n^{ -\ee}\log n, \eta_n > k \right) \notag\\
& = \bb P \left( |S_k|  > \alpha \sigma  n^{1/2 -\ee}\log n,  
\max_{1 \leq j \leq k}  |X_j|  \leq \alpha  \sigma n^{1/2-\ee}  \right)  \notag\\
& \leq 2 \exp \left\{  \frac{n^{1/2-\ee} \log n}{n^{1/2-\ee}}  \right\}  
     \left( \frac{  \gamma n^{1-2\ee}\log n }{ \alpha^2 n^{1/2 -\ee}  n^{1/2-\ee} \log n   } \right)^{  \frac{n^{1/2-\ee}\log n}{n^{1/2-\ee}}   }  \notag\\
& \leq 2  n  \left( \frac{ \gamma }{ \alpha^2  } \right)^{ \log n } 
 = 2 n^{ - (\log (\alpha^2/\gamma)-1)   } \leq 2  n^{- \log \alpha}, 
\end{align}
where in the last inequality we choose $\alpha \geq  \gamma$.
From \eqref{bound J_41-001} and \eqref{Apply-Fuk-Nagaev-aa}, 
we get the following upper bound for the first term $J_{31}$: for any $\alpha \geq e^2$ and $n\geq 2$, 
\begin{align} \label{final bound for J_31}
J_{31}  \leq  \sum_{k = 1}^{N_n}  \bb P \left( \frac{|S_{k}|}{\sigma \sqrt{n}} > \alpha n^{-\ee }\log n,  \eta_n > k \right)
\leq 2 N_n   n^{- \log \alpha}  \leq c_\gamma  n^{-\frac{1}{2}\log \alpha}. 
\end{align}

We proceed to handle the second term $J_{32}$. Taking into account that $H(x') = x' L(x')$ for $x' \in \bb R$, 
by Lemma \ref{lem-inequality for L},  we get
\begin{align*} 
J_{32}  &\leq \sum_{k = 1}^{N_n} \bb E
      \left(\frac{x+S_k}{\sigma\sqrt{n}}  L \left( \frac{x+S_k}{\sigma\sqrt{n}}  \right);  
      \tau_x > k, \eta_n \leq k,  \nu_n = k \right) \notag \\
&\leq L\left( \frac{x}{\sigma\sqrt{n}} \right)\sum_{k = 1}^{N_n} \bb E
      \left(\frac{x+S_k}{\sigma\sqrt{n}}   \left(1 + c \frac{|S_k|}{\sqrt{n}} \right);   
      \tau_x > k, \eta_n \leq k,  \nu_n = k \right).
\end{align*}
By \eqref{weak  bound for V_n-001}, we have $x + S_{k}  \leq V(x) + |S_{k}|$. 
It follows that
\begin{align} \label{E22-decom-aa}
J_{32} 
&\leq c   \frac{V(x) L\left( \frac{x}{\sigma\sqrt{n}} \right) }{\sqrt{n}} 
\sum_{k = 1}^{N_n} \bb P  \left( 
\tau_x > k, \eta_n \leq k,  \nu_n = k \right) 
      \notag\\
& \quad + c L\left( \frac{x}{\sigma\sqrt{n}} \right) \left(1+ \frac{V(x)}{\sqrt{n}}\right)  
\sum_{k = 1}^{N_n} \bb E \left( \frac{|S_k|}{\sqrt{n}};   \tau_x > k, \eta_n \leq k,  \nu_n = k \right) \notag\\ 
& \quad + c L\left( \frac{x}{\sigma\sqrt{n}} \right)
\sum_{k = 1}^{N_n} \bb E\left(   \frac{S_k^2}{ n};  \tau_x > k, \eta_n \leq k,  \nu_n = k \right) \notag\\
&=: G_1 +G_2 +G_3.
\end{align}

{\it Bound of $G_1$.} 
By the definition of $\eta_n$ (cf.\ \eqref{def eta_n-001}) and the fact that $N_n = [\gamma n^{1-2\ee} \log n]$, 
it holds that for any $\ee \in (0,\frac{1}{2}]$, 
\begin{align}\label{E22-first-part}
 \bb P \left( \eta_n \leq  N_n \right)  
\leq   N_n   \bb P \left( |X_1| >  \alpha \sigma n^{1/2 - \ee}  \right)  
\leq   c\frac{n^{1-2\ee} \bb E|X_1|^{2+\delta}}{ n^{ (1/2 -\ee) (2+\delta) } }\log n  
\leq     \frac{c}{ n^{\frac{\delta}{2} - \delta\ee} } \log n. 
\end{align} 
Using \eqref{E22-first-part} gives the following bound for $G_1$: 
\begin{align} \label{final bound G_1}
G_1  \leq c   \frac{V(x) L\left( \frac{x}{\sigma\sqrt{n}} \right) }{\sqrt{n}} \bb P \left( \eta_n \leq  N_n \right) 
\leq c   \frac{V(x) L\left( \frac{x}{\sigma\sqrt{n}} \right) }{ n^{ \frac{1}{2}+\frac{\delta}{2} - \delta\ee} } \log n.
\end{align}

{\it Bound of $G_2$.} 
Since $|S_{\nu_n}| \leq |S_{\nu_n} - S_k| + |X_k| + |S_{k-1}|$, we write
\begin{align} \label{bound with 3 F-001}
&\sum_{k = 1}^{N_n} \bb E \left( \frac{|S_k|}{\sqrt{n}};   \tau_x > k, \eta_n \leq k,  \nu_n = k \right) \notag\\ 
&= \bb{E} \left( \frac{|S_{\nu_n}|}{\sqrt{n}};  
\tau_x > \nu_n, \eta_n \leq \nu_n \leq  N_n \right)   \notag\\
& \leq \sum_{k=1}^{ N_n }  \bb{E} \left( \frac{|S_{\nu_n}|}{\sqrt{n}};  \tau_x > k,  k \leq  \nu_n \leq  N_n,  \eta_n = k \right)   \notag\\
& \leq  \frac{1}{\sqrt{n}} \sum_{k=1}^{ N_n }  
      \bb{E} \left( |S_{\nu_n} - S_k|;  \tau_x > k,  k \leq  \nu_n \leq  N_n,  \eta_n = k \right)    \notag\\
& \quad  +   \frac{1}{\sqrt{n}} \sum_{k=1}^{ N_n }  
      \bb{E} \left( |X_k|;  \tau_x > k,  k \leq  \nu_n \leq  N_n,  \eta_n = k \right)   \notag\\
& \quad  +   \frac{1}{\sqrt{n}} \sum_{k=1}^{ N_n }  
      \bb{E} \left( |S_{k-1}|;  \tau_x > k,  k \leq  \nu_n \leq  N_n,  \eta_n = k \right)   \notag\\
 & =:  F_1 + F_2 + F_3.
\end{align}
For $F_1$, we have 
\begin{align} \label{bound FFF_1-001}
F_1  
& \leq  \frac{1}{\sqrt{n}} \sum_{k=1}^{ N_n }  
      \bb{E} \left( \max_{ l \in [k, N_n ] } |S_l - S_k|;  \tau_x > k,  \eta_n = k \right)   \notag\\
& \leq  \frac{1}{\sqrt{n}} \sum_{k=1}^{ N_n }  
      \bb{E} \left( \max_{ l \in [k, N_n ] } |S_l - S_k|;  \tau_x > k - 1,  \eta_n = k \right)   \notag\\      
& =   \frac{1}{\sqrt{n}} \sum_{k=1}^{ N_n }  
      \bb{E} \left( \max_{ l \in [k, N_n ] } |S_l - S_k|  \right)
      \bb P \left( \tau_x > k - 1 \right)   \bb P \left( \eta_n = k \right) \notag\\
& \leq  \frac{c\sqrt{N_n}}{\sqrt{n}}   
    \sum_{k=1}^{ N_n }  \bb P \left( \tau_x > k - 1 \right) \bb P \left( \eta_n = k \right),
\end{align}
where for the last line we used Burkholder's inequality (see Lemma \ref{Lem-Burkholder}): 
\begin{align*} 
\bb{E} \left( \max_{ l \in [k, N_n ] } |S_l - S_k|  \right) \leq c\sqrt{N_n}.
\end{align*}
By the definition of $\eta_n$ (cf.\ \eqref{def eta_n-001}), and 
using the fact that $\bb E (|X_1|^{2 + \delta}) < \infty$, we have
 \begin{align} \label{bound FFF_1-002}
\bb P\left( \eta_n = k  \right) = \bb P \left( |X_k| > \alpha \sigma n^{1/2-\ee} \right)
\leq  \frac{c}{ n^{( 1/2-\ee ) (2 + \delta)} } =  \frac{c}{ n^{1 + \frac{\delta}{2}  -2 \ee  -  \ee\delta} }. 
\end{align}
 Since the function $t\mapsto L(t)$ 
 is decreasing on $\bb R_+$, we have
 $L\left(\frac{x}{\sigma\sqrt{k}}\right) \leq L\left(\frac{x}{\sigma\sqrt{n}}\right)$ for any $1 \leq k \leq n$ and $x\geq 0$.
 Therefore, by Lemma \ref{new bound for Ptau_x>n-001}, there exists a constant $c>0$ such that, 
 for any $1 \leq k \leq n$ and $x\geq 0$, 
 \begin{align} \label{bound FFF_1-002a for tau_x}
 \bb P \left( \tau_x > k \right) 
 \leq c \frac{1+V(x) L\left(\frac{x}{\sigma\sqrt{k}}\right) }{k^{1/2}} 
 \leq c \frac{1+V(x) L\left(\frac{x}{\sigma\sqrt{n}}\right) }{k^{1/2}}.
\end{align}
 For $x< 0$, in view of \eqref{proba-tau-x-001}, we have 
 $ \bb P \left( \tau_x > k \right)\leq  \bb P \left( \tau_0 > k \right) 
 \leq  \frac{c}{k^{1/2}} \leq c \frac{1+V(x) L(\frac{x}{\sigma\sqrt{n}}) }{k^{1/2}}$, 
 which proves that \eqref{bound FFF_1-002a for tau_x} holds for all $x\in \bb R$. 
Therefore, recalling that $N_n =[\gamma n^{1-2\ee} \log n]$, we get
\begin{align} \label{bound FFF_1-003}
\sum_{k=1}^{ N_n }  \bb P \left( \tau_x > k - 1 \right)
&\leq  1 + c \left( 1+V(x) L\left(\frac{x}{\sigma\sqrt{n}}\right) \right)  \sum_{k=1}^{ N_n }  
\frac{ 1 }{ k^{1/2} } \notag \\
&\leq  1 + c \left(1+V(x) L\left(\frac{x}{\sigma\sqrt{n}}\right)\right)  N_n^{\frac{1}{2}}  \notag \\
&\leq  c  (n^{1 -  2\ee} \log n)^{\frac{1}{2} }  \left( 1 + V(x) L\left(\frac{x}{\sigma\sqrt{n}}\right) \right). 
\end{align}
Combining \eqref{bound FFF_1-002} and \eqref{bound FFF_1-003}, we get 
\begin{align}\label{ine-sum-eta-tau}
\sum_{k=1}^{ N_n }  \bb P \left( \tau_x > k - 1 \right) \bb P \left( \eta_n = k \right)
\leq  c \frac{\sqrt{\log n}}{ n^{\frac{1}{2} + \frac{\delta}{2}  - \ee  -  \ee\delta} }  \left( 1 + V(x) L\left(\frac{x}{\sigma\sqrt{n}}\right) \right). 
\end{align}
From \eqref{bound FFF_1-001} and \eqref{ine-sum-eta-tau}, 
we obtain the following bound for $F_1$: 
\begin{align}\label{final bound for F_1-001}
F_1 &\leq  \frac{c}{\sqrt{n}}   (n^{1 - 2\ee}\log n)^{1/2}  \frac{\sqrt{\log n}}{ n^{\frac{1}{2} + \frac{\delta}{2}  - \ee  -  \ee\delta} } 
  \left( 1 + V(x) L\left(\frac{x}{\sigma\sqrt{n}}\right) \right) \notag \\
&\leq c \frac{ \log n}{ n^{ \frac{1}{2} + \frac{\delta}{2} - \delta\ee  }  }  \left( 1 + V(x) L\left(\frac{x}{\sigma\sqrt{n}}\right) \right). 
\end{align}

Now we deal with the term $F_2$ defined in \eqref{bound with 3 F-001}. By independence, we have
\begin{align*}
F_2 & \leq  \frac{1}{\sqrt{n}} \sum_{k=1}^{ N_n }  
      \bb{E} \left( |X_k|;  \tau_x > k - 1,  |X_k| > \alpha \sigma n^{1/2 - \ee} \right)   \notag\\
& =  \frac{1}{\sqrt{n}}  \bb{E} \left( |X_1|;   |X_1| > \alpha \sigma n^{1/2 - \ee} \right) 
    \sum_{k=1}^{ N_n }   \bb{P} \left( \tau_x > k - 1 \right). 
\end{align*}
Using the condition $\bb E (|X_1|^{2 + \delta}) < \infty$ 
and the bound \eqref{bound FFF_1-003},
we get that, for any $n \geq 1$ and $x \in \bb R$,  
\begin{align} \label{bound final F_2-001}
F_2 
&\leq  \frac{c}{\sqrt{n}} \frac{1}{ n^{( 1/2-\ee ) (1 + \delta)} }  (n^{1 -  2\ee} \log n)^{\frac{1}{2}}   
 \left( 1 + V(x) L\left(\frac{x}{\sigma\sqrt{n}}\right) \right) \notag \\
&\leq  \frac{c \sqrt{\log n}}{ n^{\frac{1}{2}+ \frac{\delta}{2} -\delta\ee} }  \left( 1 + V(x) L\left(\frac{x}{\sigma\sqrt{n}}\right) \right).
\end{align}

To handle $F_3$ in \eqref{bound with 3 F-001}, we decompose it into two terms:
\begin{align}\label{bound-decom-F3}
F_3 & \leq  \frac{1}{\sqrt{n}} \sum_{k=1}^{ N_n }  
      \bb{E} \left( |S_{k-1}|;  |S_{k-1}|  > \alpha  \sigma n^{1/2 - \ee} \log n,   \eta_n = k \right)   \notag\\
   & \quad  +  \frac{1}{\sqrt{n}} \sum_{k=1}^{ N_n }  
      \bb{E} \left( |S_{k-1}|;  |S_{k-1}|  \leq  \alpha \sigma n^{1/2 - \ee} \log n,  \tau_x > k,   \eta_n = k \right)   \notag\\
   & =: F_{31} + F_{32}, 
\end{align}
where $\alpha \geq e$ is the same as that in \eqref{def eta_n-001}. 
For the first term $F_{31}$ in \eqref{bound-decom-F3}, 
using the fact that $|S_{k-1}| \leq n^{3/2}$ on the event $\{ \eta_n = k \}$ for any $1 \leq k \leq N_n$,
 and the bound \eqref{Apply-Fuk-Nagaev-aa},  we get 
\begin{align}\label{bound-final-F31}
F_{31} & \leq c n  \sum_{k=1}^{ N_n }  
      \bb{P} \left( |S_{k-1}|  > \alpha \sigma  n^{1/2 - \ee}\log n,   \eta_n = k \right)   \notag\\
& \leq c n \sum_{k=1}^{ N_n }  
      \bb{P} \left( |S_{k-1}|  > \alpha \sigma  n^{1/2 - \ee}\log n,  \max_{1 \leq j \leq k-1} |X_j| \leq \alpha \sigma  n^{1/2-\ee} \right)   \notag\\
& \leq   c n^2 n^{- \log \alpha} \leq c n^{- \frac{1}{2} \log \alpha}, 
\end{align}
where in the last inequality we take $\alpha \geq e^3$. 
For the second term $F_{32}$ in \eqref{bound-decom-F3}, 
using \eqref{ine-sum-eta-tau}, we have 
\begin{align}\label{bound-final-F32}
F_{32}  & \leq  c   n^{- \ee }\log n \sum_{k=1}^{ N_n }  
      \bb{P} \left(  \tau_x > k,   \eta_n = k \right)    \notag\\
    &  \leq  c   n^{- \ee } \log n \sum_{k=1}^{ N_n }  
      \bb P \left( \tau_x > k - 1 \right)   \bb P \left( \eta_n = k \right)  \notag\\
      &\leq  c n^{ - \ee}\log n   \frac{\sqrt{\log n}}{ n^{\frac{1}{2} + \frac{\delta}{2}  - \ee  -  \ee\delta} }
\left( 1 + V(x) L\left(\frac{x}{\sigma\sqrt{n}}\right) \right) \notag \\
&\leq  c \frac{ (\log n)^{3/2}}{ n^{\frac{1}{2} + \frac{\delta}{2}  - \delta\ee  } }  \left( 1 + V(x) L\left(\frac{x}{\sigma\sqrt{n}}\right) \right).
\end{align}
Substituting \eqref{bound-final-F31} and \eqref{bound-final-F32} into \eqref{bound-decom-F3},
and taking $\alpha>0$ to be sufficiently large (e.g. $\log \alpha > 1 + \delta - 2 \delta \ee$), 
we get the following bound for $F_3$: 
\begin{align} \label{bound F3-final-010}
F_3 \leq  \frac{c (\log n)^{3/2}}{ n^{\frac{1}{2} + \frac{\delta}{2}  - \delta\ee   } }  \left( 1 + V(x) L\left(\frac{x}{\sigma\sqrt{n}}\right) \right).
\end{align}
Coming back to estimation of $G_2$ from \eqref{E22-decom-aa} and collecting the bounds for $F_1$, $F_2$, $F_3$
from \eqref{bound with 3 F-001},
\eqref{final bound for F_1-001}, \eqref{bound final F_2-001} and \eqref{bound F3-final-010}, we get that 
there exists $c=c_{\ee}>0$ such that
\begin{align} \label{bbb bound with G_2-001}
G_2  & \leq c  L\left( \frac{x}{\sigma\sqrt{n}} \right) \left(1+ \frac{V(x)}{\sqrt{n}}\right)  (F_1 + F_2 + F_3 ) \notag\\ 
& \leq c L\left( \frac{x}{\sigma\sqrt{n}} \right) \left(1+ \frac{V(x)}{\sqrt{n}}\right)  
\frac{1 + V(x) L\left(\frac{x}{\sigma\sqrt{n}}\right)}{ n^{\frac{1}{2} + \frac{\delta}{2}  - \delta\ee   } } (\log n)^{3/2}.
\end{align}

{\it Bound of $G_3$.} 
The estimate of $G_3$ can be obtained by following the same strategy of the proof of $G_2$ in \eqref{bbb bound with G_2-001}. 
As in \eqref{bound with 3 F-001}, using the inequality 
$S_{\nu_n}^2 \leq 3(S_{\nu_n} - S_k)^2 + 3X_k^2 + 3S_{k-1}^2$, we have
\begin{align} \label{bbb bound with G_3-001}
&\sum_{k = 1}^{N_n} \bb E\left(   \frac{S_k^2}{ n};  \tau_x > k, \eta_n \leq k,  \nu_n = k \right) \notag\\
&=  \bb{E} \left( \frac{S_{\nu_n}^2}{n};  
\tau_x > \nu_n, \eta_n \leq \nu_n \leq  N_n \right)   \notag\\
& \leq \sum_{k=1}^{ N_n }  
      \bb{E} \left( \frac{S_{\nu_n}^2}{n};  \tau_x > k,  k \leq  \nu_n \leq  N_n,  \eta_n = k \right)   \notag\\
& \leq  \frac{3}{n} \sum_{k=1}^{ N_n }  
      \bb{E} \left( (S_{\nu_n} - S_k)^2;  \tau_x > k,  k \leq  \nu_n \leq  N_n,  \eta_n = k \right)    \notag\\
& \quad  +   \frac{3}{n} \sum_{k=1}^{ N_n }  
      \bb{E} \left( X_k^2;  \tau_x > k,  k \leq  \nu_n \leq  N_n,  \eta_n = k \right)   \notag\\
& \quad  +   \frac{3}{n} \sum_{k=1}^{ N_n }  
      \bb{E} \left( S_{k-1}^2;  \tau_x > k,  k \leq  \nu_n \leq  N_n,  \eta_n = k \right)   \notag\\
 & =:  F'_1 + F'_2 + F'_3.
\end{align}
For $F'_1$, following the proof of \eqref{bound FFF_1-001}, using Lemma \ref{Lem-Burkholder}
and \eqref{ine-sum-eta-tau}, we obtain the following bound for $F'_1$: 
\begin{align} \label{bbb final bound for F_1-001}
F'_1  
& \leq   \frac{3}{n} \sum_{k=1}^{ N_n }  
      \bb{E} \left( \max_{ l \in [k, N_n ] } (S_l - S_k)^2  \right)
      \bb P \left( \tau_x > k - 1 \right)   \bb P \left( \eta_n = k \right) \notag\\
& \leq  c \frac{N_n }{n}    
    \sum_{k=1}^{ N_n }  \bb P \left( \tau_x > k - 1 \right) \bb P \left( \eta_n = k \right) \notag\\
&\leq  c n^{-2\ee} \log n  \frac{\sqrt{\log n}}{ n^{\frac{1}{2} + \frac{\delta}{2}  - \ee  -  \ee\delta} } 
  \left( 1 + V(x) L\left(\frac{x}{\sigma\sqrt{n}}\right) \right) \notag \\
&\leq  c \frac{(\log n)^{3/2}}{ n^{ \frac{1}{2} + \frac{\delta}{2} +\ee - \delta\ee  }  }  \left( 1 + V(x) L\left(\frac{x}{\sigma\sqrt{n}}\right) \right).
\end{align}

Now we deal with the second term $F'_2$ in \eqref{bbb bound with G_3-001}. By independence, we have
\begin{align*}
F'_2 
& \leq  \frac{c}{n} \sum_{k=1}^{ N_n }  
      \bb{E} \left( X_k^2;  \tau_x > k - 1,  |X_k| > \alpha\sigma n^{1/2 - \ee} \right)   \notag\\
& =  \frac{c}{n}  \bb{E} \left( X_1^2;   |X_1| > \alpha\sigma n^{1/2 - \ee} \right) 
    \sum_{k=1}^{ N_n }   \bb{P} \left( \tau_x > k - 1 \right). 
\end{align*}
Using the condition $\bb E (|X_1|^{2 + \delta}) < \infty$ 
and the bound \eqref{bound FFF_1-003},
we get the following bound for $F'_2$: 
\begin{align} \label{bbb bound final F_2-001}
F'_2 
&\leq  \frac{c}{n} \frac{1}{ n^{( 1/2-\ee ) \delta} }  (n^{1 -  2\ee} \log n)^{\frac{1}{2}}  \left( 1 + V(x) L\left(\frac{x}{\sigma\sqrt{n}}\right) \right) \notag \\
&\leq  
 c \frac{(\log n)^{1/2}}{ n^{\frac{1}{2} + \frac{\delta}{2} + \ee -\delta\ee } }  \left( 1 + V(x) L\left(\frac{x}{\sigma\sqrt{n}}\right) \right). 
\end{align}

To bound the third term $F'_3$ in \eqref{bbb bound with G_3-001}, as in \eqref{bound-decom-F3}, we decompose it into two terms:
\begin{align}\label{bbb bound-decom-F3}
F'_3 & \leq  \frac{c}{n} \sum_{k=1}^{ N_n }  
      \bb{E} \left( S_{k-1}^2;  |S_{k-1}|  >  \alpha \sigma n^{1/2 - \ee} \log n,   \eta_n = k \right)   \notag\\
   & \quad  + \frac{c}{n} \sum_{k=1}^{ N_n }  
      \bb{E} \left( S_{k-1}^2;  |S_{k-1}|  \leq \alpha \sigma n^{1/2 - \ee} \log n,  \tau_x > k,   \eta_n = k \right)   \notag\\
   & =: F'_{31} + F'_{32}. 
\end{align}
For $F'_{31}$, using the fact that $S_{k-1}^2 \leq n^{3}$ on the event $\{ \eta_n = k \}$ for any $1 \leq k \leq N_n$, 
 and the bound \eqref{Apply-Fuk-Nagaev-aa}, 
we get that, for $\alpha \geq e^6$,
\begin{align}\label{bbb bound-final-F31}
F'_{31}  \leq c n^2 \sum_{k=1}^{ N_n }  
      \bb{P} \left( |S_{k-1}|  > \alpha \sigma n^{1/2 - \ee} \log n,   \eta_n = k \right) 
      \leq c n^2 N_n  n^{- \log \alpha}
       \leq  c n^{- \frac{1}{2}\log \alpha}.
\end{align}
For $F'_{32}$, using \eqref{ine-sum-eta-tau}, we have 
\begin{align} \label{bbb bound-final-F32}
F'_{32}   \leq  c   n^{ - 2\ee } \sum_{k=1}^{ N_n }  
  \bb P \left( \tau_x > k - 1 \right)   \bb P \left( \eta_n = k \right)
\leq  c \frac{\sqrt{\log n}}{ n^{ \frac{1}{2} + \frac{\delta}{2} +\ee  - \delta\ee  } }  \left( 1 + V(x) L\left(\frac{x}{\sigma\sqrt{n}}\right) \right).
\end{align}
Substituting \eqref{bbb bound-final-F31}
and \eqref{bbb bound-final-F32} into \eqref{bbb bound-decom-F3}, 
and taking $\alpha>0$ to be sufficiently large (e.g. $\log \alpha > 1 + \delta + 2 \ee - 2 \delta \ee$), 
we obtain 
\begin{align} \label{bbb bound F3-final-010}
F'_3
\leq c \frac{\sqrt{\log n}}{ n^{\frac{1}{2} + \frac{\delta}{2} + \ee  - \delta\ee   } }  \left( 1 + V(x) L\left(\frac{x}{\sigma\sqrt{n}}\right) \right).
\end{align}

Coming back to estimation of $G_3$ from \eqref{E22-decom-aa} and collecting the bounds for $F'_1$, $F'_2$, $F'_3$
from \eqref{bbb bound with G_3-001},
\eqref{bbb final bound for F_1-001}, \eqref{bbb bound final F_2-001} and \eqref{bbb bound F3-final-010}, we get
\begin{align} \label{final bound G_3-001}
G_3  \leq c  L\left( \frac{x}{\sigma\sqrt{n}} \right)  (F'_1 + F'_2 + F'_3 ) 
 \leq c L\left( \frac{x}{\sigma\sqrt{n}} \right)   
\frac{1 + V(x) L\left(\frac{x}{\sigma\sqrt{n}}\right)}{ n^{\frac{1}{2} + \frac{\delta}{2} +\ee  - \delta\ee   } }  (\log n)^{3/2}.
\end{align}

The bound of $J_{32}$ is obtained by putting together the founds for $G_1$, $G_2$ and $G_3$
(cf.\ \eqref{E22-decom-aa}, \eqref{final bound G_1}, \eqref{bbb bound with G_2-001} and \eqref{final bound G_3-001}): 
\begin{align*} 
J_{32} 
&\leq  G_1 + G_2 + G_3 \notag \\
& \leq    c   \frac{ V(x) L\left( \frac{x}{\sigma\sqrt{n}} \right)  }{ n^{ \frac{1}{2}+\frac{\delta}{2} - \delta\ee }} \log n \notag \\
& \quad +  c L\left( \frac{x}{\sigma\sqrt{n}} \right) \left(1+ \frac{V(x)}{\sqrt{n}}\right)  
\frac{1 + V(x) L\left(\frac{x}{\sigma\sqrt{n}}\right)}{ n^{\frac{1}{2} + \frac{\delta}{2}  - \delta\ee  } }  (\log n)^{3/2}  \notag \\
& \quad + c L\left( \frac{x}{\sigma\sqrt{n}} \right)   
\frac{1 + V(x) L\left(\frac{x}{\sigma\sqrt{n}}\right)}{ n^{\frac{1}{2} + \frac{\delta}{2} +\ee  - \delta\ee   } }  (\log n)^{3/2}  \notag \\
& \leq c \frac{1 + V(x) L\left(\frac{x}{\sigma\sqrt{n}}\right)}{ n^{\frac{1}{2} + \frac{\delta}{2}  - \delta\ee  } }  (\log n)^{3/2}, 
\end{align*}
where in the last inequality we used the fact $L$ is bounded on $\bb R$ and 
$V(x)L\left(  \frac{x}{\sigma \sqrt{n}} \right) \leq  c_2 \sqrt{n}$
(cf.\ \eqref{lower bound of VL-001}). 
Together with the bound \eqref{final bound for J_31} where $\alpha >0$ is chosen to be large enough, 
this concludes the proof of Lemma \ref{Bound for J_4}. 
\end{proof}

Next, we establish upper and lower bounds for the main term $J_{4}$, as defined in \eqref{term-J5-001}.

\begin{lemma}\label{Lem-bound-UL-J4}
For any  $\ee \in [0,\frac{1}{2}]$, there exists a constant $c=c_{\ee}>0$ 
such that for any $n\geq 1$ and $x\in \bb R$,
\begin{align} \label{final bound for J4-001bb}
J_4  \leq \frac{V(x) L\left(\frac{x}{\sigma\sqrt{n}} \right) }{\sigma\sqrt{n}}  \left( 1+ c n^{-\ee} \log n  \right). 
\end{align}
Moreover, choosing $\ee = \frac{\delta}{4(3+\delta)} < \frac{1}{4}$,  
we have that there exists a constant $c>0$ such that for any $n\geq 1$ and $x\in \bb R$,
\begin{align} \label{final bound for J4-002}
J_4 \geq 
 \frac{V(x) L\left(\frac{x}{\sigma\sqrt{n}} \right) }{\sigma\sqrt{n}} 
\left( 1- c n^{- \frac{\delta}{4(3+\delta)}} \log n  \right)
- \frac{c}{n^{\frac{1}{2}+ \frac{\delta}{4}}  }. 
\end{align}
\end{lemma}

\begin{proof}
Since $H(x')=x'L(x')$ for any $x'\in \bb R$, by \eqref{term-J5-001}, we have 
\begin{align} \label{Equality-J4-func-L}
J_{4}  =  \sum_{k = 1}^{N_n}  
\bb E \Bigg(\frac{x+S_k }{\sigma\sqrt{n}}  L\left(\frac{x+S_k  }{\sigma\sqrt{n}} \right); 
    \frac{|S_k|}{\sigma \sqrt{n}} \leq \alpha n^{-\ee}\log n, \tau_x > k, \nu_n = k\Bigg).
\end{align}
By Lemmata \ref{lem-inequality for L} and \ref{lemma E1 and E2},
for any $\ee \in  [0, \frac{1}{2}]$, 
there exists a constant $c>0$ such that for any $n \geq 1$ and $x \in \bb R$, 
\begin{align*} 
J_{4} 
&\leq \frac{ L\left(\frac{x}{\sigma\sqrt{n}} \right) \left( 1+ c n^{-\ee}\log n  \right)}{\sigma\sqrt{n}} 
\sum_{k = 1}^{N_n} 
\bb E \left( x+S_k ;   \tau_x > k, \nu_n = k\right) \notag \\
&=  \frac{L\left(\frac{x}{\sigma\sqrt{n}} \right) \left( 1+  c n^{-\ee} \log n \right)}{\sigma\sqrt{n}} 
\bb E \left( x+S_{\nu_n} ;   \tau_x > \nu_n, \nu_n \leq N_n \right) \notag \\ 
&\leq \frac{V(x) L\left(\frac{x}{\sigma\sqrt{n}} \right) }{\sigma\sqrt{n}}  \left( 1+ c n^{-\ee} \log n  \right), 
\end{align*}
which proves \eqref{final bound for J4-001bb}. 

Now we proceed to give a lower bound for $J_4$. 
By Lemma \ref{lem-inequality for L}, for $1\leq k\leq N_n$, it holds
\begin{align*} 
L\left(\frac{x+S_k  }{\sigma\sqrt{n}} \right) 
\geq L\left(\frac{x}{\sigma\sqrt{n}} \right) \left(1 -  c \frac{|S_k|}{\sqrt{n}}\right). 
\end{align*}
Substituting this into \eqref{Equality-J4-func-L} leads to 
\begin{align} \label{lower bound J_4-000}
J_{4} 
&\geq \sum_{k = 1}^{N_n} \frac{L\left(\frac{x}{\sigma\sqrt{n}} \right) }{\sigma\sqrt{n}} 
\bb E \left( \left(x+S_k\right) \left(1-  c \frac{|S_k|}{\sqrt{n}}\right);  \frac{|S_k|}{\sigma \sqrt{n}} \leq \alpha n^{-\ee}\log n, 
\tau_x > k, \nu_n = k\right)  \notag\\
&\geq \frac{ \left( 1 - c n^{-\ee}\log n  \right)  L\left(\frac{x}{\sigma\sqrt{n}} \right) }{\sigma\sqrt{n}} 
\sum_{k = 1}^{N_n} 
\bb E \left( x+S_k ;   \tau_x > k, \nu_n = k\right) \notag\\
& \quad - \frac{ L\left(\frac{x}{\sigma\sqrt{n}} \right)}{\sigma\sqrt{n}} 
\sum_{k = 1}^{N_n} 
\bb E \left( x+S_k ;  \frac{|S_k|}{\sigma \sqrt{n}} > \alpha n^{-\ee}\log n,  \tau_x > k, \nu_n = k\right) \notag \\
&=:  J_{5} - J_{6}. 
\end{align}
Using Lemma \ref{lemma E1 and E2},  we derive that 
\begin{align*}
\sum_{k = 1}^{N_n} \bb E \left( x+S_k ;   \tau_x > k, \nu_n = k\right)
& = \mathbb{E} \left( x + S_{\nu_n}; \tau_x > \nu_n, \nu_n \leq N_n \right)  \notag\\
& \geq  \left( 1  -  c  n^{-(1/2 - \ee) \delta} \right) V(x). 
\end{align*}
Substituting this into $J_5$ gives 
\begin{align} \label{final bound J_5 001}
J_5 
&\geq \left( 1 - c n^{-\ee}\log n  \right) \left( 1  -  c' n^{-(1/2 - \ee) \delta} \right) \frac{ V(x) L\left(\frac{x}{\sigma\sqrt{n}} \right) }{\sigma\sqrt{n}} \notag\\ 
&\geq \frac{ V(x) L\left(\frac{x}{\sigma\sqrt{n}} \right) }{\sigma\sqrt{n}}  \left( 1 - c n^{-\ee}\log n  -  c' n^{-(1/2 - \ee) \delta} \right). 
\end{align}
Let $\eta_n$ be defined as in \eqref{def eta_n-001}. 
For $J_{6}$, by \eqref{lower bound J_4-000} we have 
\begin{align*} 
J_{6} 
&= \frac{ L\left(\frac{x}{\sigma\sqrt{n}} \right)}{\sigma\sqrt{n}} \sum_{k = 1}^{N_n} \bb E \left( x+S_k ;  \frac{|S_k|}{\sigma \sqrt{n}} > \alpha n^{-\ee}\log n, \tau_x > k, \eta_n > k, \nu_n = k\right) \notag  \\
& \quad + \frac{ L\left(\frac{x}{\sigma\sqrt{n}} \right)}{\sigma\sqrt{n}} \sum_{k = 1}^{N_n} \bb E \left( x+S_k ;  \frac{|S_k|}{\sigma \sqrt{n}} > \alpha n^{-\ee}\log n,  \tau_x > k, \eta_n \leq k,  \nu_n = k\right) \\
&=: J_{61} + J_{62}. 
\end{align*}
For $1\leq k\leq N_n=[\gamma n^{1-2\ee}\log n]$, on the set $\{\eta_n > k\}$ 
it holds that $\{ \max_{1 \leq j \leq k} |X_j| \leq \sigma \alpha n^{1/2-\ee}  \}$ and hence $|S_k| \leq c n^{3/2}$.
Since $x + S_{k} \leq  V(x) + |S_{k}|$, using \eqref{Apply-Fuk-Nagaev-aa} and choosing $\alpha$ large enough, we get
\begin{align} \label{upper bound for J_61-0012}
J_{61} 
&\leq \frac{ L\left(\frac{x}{\sigma\sqrt{n}} \right)}{\sigma\sqrt{n}} \sum_{k = 1}^{N_n} \bb E \left( V(x)+|S_k| ;  \frac{|S_k|}{\sigma \sqrt{n}} > \alpha n^{-\ee}\log n, \eta_n > k, \tau_x > k, \nu_n = k\right) \notag \\
&\leq \frac{ L\left(\frac{x}{\sigma\sqrt{n}} \right) \left(V(x)+cn^{3/2}\right)  }{\sigma\sqrt{n}}  
 \sum_{k = 1}^{N_n} \bb P \left(  \frac{|S_k|}{\sigma \sqrt{n}} > \alpha n^{-\ee}\log n, 
  \max_{1 \leq j \leq k} |X_j| \leq \sigma n^{1/2-\ee} \right) \notag  \\
&\leq cL\left(\frac{x}{\sigma\sqrt{n}} \right)\left(1+V(x)\right) n N_n  n^{- \log \alpha} \notag  \\
&\leq cL\left(\frac{x}{\sigma\sqrt{n}} \right)(1+V(x)) n^{- \frac{1}{2} \log \alpha}.
\end{align}
Taking into account that $x + S_{\nu_n} \leq  \max\{x,0\} + |S_{\nu_n}| \leq V(x) + |S_{\nu_n}|$, we have
\begin{align*} 
J_{62} &\leq \frac{ L\left(\frac{x}{\sigma\sqrt{n}} \right)}{\sigma\sqrt{n}} \sum_{k = 1}^{N_n}  \bb E \left( x+S_{k};  
\tau_x > k, \eta_n \leq k, \nu_n =k \right) \\
&\leq \frac{ V(x) L\left(\frac{x}{\sigma\sqrt{n}} \right)}{\sigma\sqrt{n}} \sum_{k = 1}^{N_n} 
\bb P \left( \tau_x > k,  \eta_n \leq k,   \nu_n = k\right) \\
& \quad + \frac{ L\left(\frac{x}{\sigma\sqrt{n}} \right)}{\sigma\sqrt{n}} \sum_{k = 1}^{N_n} \bb E \left( |S_{\nu_n}| ;  
\tau_x > k, \eta_n \leq k,  \nu_n = k\right) \\
& \leq G_1 + G_2,
\end{align*}
where $G_1$ and $G_2$ are defined by \eqref{E22-decom-aa}. 
Collecting the bounds \eqref{upper bound for J_61-0012}, \eqref{final bound G_1} and \eqref{bbb bound with G_2-001}, 
we get
\begin{align} \label{final bound of J_6 001-00a}
J_6 & \leq  cL\left(\frac{x}{\sigma\sqrt{n}} \right)(1+V(x)) n^{-\frac{1}{2}\log \alpha}  \notag \\
& \quad + c   \frac{V(x) L\left( \frac{x}{\sigma\sqrt{n}} \right)  }{ n^{ \frac{1}{2}+\frac{\delta}{2} - \delta\ee} }  \log n
 + c L\left( \frac{x}{\sigma\sqrt{n}} \right) \left(1+ \frac{V(x)}{\sqrt{n}}\right)  
\frac{1 + V(x) L\left(\frac{x}{\sigma\sqrt{n}}\right)}{ n^{\frac{1}{2} + \frac{\delta}{2}  - \delta\ee   } } (\log n)^{3/2}. 
\end{align}
Using the fact that $V(x) L ( \frac{x}{\sigma\sqrt{n}} ) \leq c\sqrt{n}$ for any $x\in \bb R$ (cf.\ \eqref{lower bound of VL-001}), 
and choosing $\ee = \frac{\delta}{4(3+\delta)} < \frac{1}{4}$ and $\alpha>0$ sufficiently large, we get
\begin{align} \label{final bound of J_6 001}
J_6 \leq  c \frac{1+V(x) L\left(\frac{x}{\sqrt{n}} \right) }{n^{\frac{1}{2} + \frac{\delta}{4}} }.
\end{align}
Since $\ee = \frac{\delta}{4(3+\delta)}$,
from \eqref{lower bound J_4-000}, \eqref{final bound J_5 001} and \eqref{final bound of J_6 001}, we obtain
\begin{align} \label{low bound J_4-001}
J_{4} 
\geq J_5 -J_6 
 \geq \frac{V(x) L\left(\frac{x}{\sigma\sqrt{n}} \right) }{\sigma\sqrt{n}} 
- c \frac{V(x) L\left(\frac{x}{\sigma\sqrt{n}} \right) }{ n^{\frac{1}{2} + \frac{\delta}{4(3+\delta)}}   } \log n  
- \frac{c}{n^{\frac{1}{2}+ \frac{\delta}{4}}  }, 
\end{align}
which ends the proof of \eqref{final bound for J4-002}. 
\end{proof}

\begin{proof}[End of the proof of Theorem \ref{Theor-probtauUN-001}]
From \eqref{bound J_3-001c}, Lemma \ref{Bound for J_4} and \eqref{final bound for J4-001bb} of Lemma \ref{Lem-bound-UL-J4}, 
we derive an upper bound for $J_2$: 
for any $\ee \in (0, \frac{1}{2}]$,  there exists a constant $c=c_{\ee,\delta}$ such that 
\begin{align} \label{upper final for J_3 001}
J_{2}  & \leq \left(  1 + c n^{\ee} r_{n,\delta} + c\frac{N_n}{n}  \right) (J_{3} + J_{4}) \notag \\
 & \leq J_4 + c J_3 + c \left(n^{\ee} r_{n,\delta} + \frac{N_n}{n}  \right) J_{4} \notag \\
&\leq 
\frac{V(x) L\left(\frac{x}{\sigma\sqrt{n}} \right) }{\sigma\sqrt{n}} 
+  c\frac{V(x) L\left(\frac{x}{\sigma \sqrt{n}} \right) }{n^{\frac{1}{2} +\ee} } \log n
+  c\frac{ 1 + V(x) L\left(\frac{x}{\sigma\sqrt{n}} \right) }{ n^{\frac{1}{2} +  \frac{\delta}{2} - \delta\ee  }  } (\log n)^{3/2}  \notag\\
& \quad + c  \left(   n^{\ee} r_{n,\delta} + \frac{N_n}{n}   \right) \frac{V(x) L\left(\frac{x}{\sigma\sqrt{n}} \right) }{\sqrt{n}}, 
\end{align}
where $r_{n,\delta}=n^{-\frac{\delta}{2(3+\delta)}}$ and $N_n=[\gamma n^{1-2\ee}\log n]$. 
Letting  $\ee=\frac{\delta}{4(3+\delta)} < \frac{1}{4}$,
using the fact that $V(x) L ( \frac{x}{\sigma\sqrt{n}} ) \leq c\sqrt{n}$ for any $x\in \bb R$ (cf.\ \eqref{lower bound of VL-001}), 
 and choosing $\alpha > 0$ to be large enough, 
we obtain
\begin{align} \label{Final upper bnd for J_2-002}
J_{2} \leq  \frac{V(x) L\left(\frac{x}{\sigma\sqrt{n}} \right) }{\sigma\sqrt{n}} 
 + \frac{c}{ n^{\frac{1}{2} +  \frac{\delta}{4} }} 
 +  c\frac{V(x) L\left(\frac{x}{\sigma\sqrt{n}} \right) }{n^{\frac{1}{2} + \frac{\delta}{4(3+\delta)}} }  \log n.   
\end{align}
Substituting \eqref{upper bound-J1-001} and \eqref{Final upper bnd for J_2-002} into \eqref{start bndtau-001},
and choosing $\gamma >0$ to be large enough, 
we get the following upper bound: 
\begin{align} \label{Final upper bnd for Ptau_x-002}
\mathbb{P} \left( \tau_x > n \right) - \frac{V(x) L\left(\frac{x}{\sigma\sqrt{n}} \right) }{\sigma\sqrt{n}} 
 \leq \frac{c}{ n^{\frac{1}{2} +  \frac{\delta}{4} }} +  c\frac{V(x) L\left(\frac{x}{\sigma\sqrt{n}} \right) }{n^{\frac{1}{2} + \frac{\delta}{4(3+\delta)}} }  \log n,   
\end{align}
 where $c>0$ depends on $\delta$.

Now we procced to give a lower bound for $J_{2}$. 
In the same way as in the proof of \eqref{bound J_3-001a}, 
using Lemma \ref{Lemma 2-KMTrate} with $a=n^{-\ee}$ and the fact that $H$ is increasing and non-negative on $\bb R_+$,
 we have that, for any $\ee \in (0, \frac{\delta}{2(3+\delta)})$,
\begin{align} \label{low-bound J_2-001}
J_{2} &\geq 
 \left(  1 - c n^{\ee} r_{n,\delta}  \right)  
 \sum_{k = 1}^{N_n} \bb E \left(    H \left( \frac{x+S_k}{\sigma\sqrt{n-k}}\right); \tau_x > k, \nu_n = k  \right) \notag \\ 
&\geq  \left(  1 - c n^{\ee} r_{n,\delta}  \right)  
 \sum_{k = 1}^{N_n} \bb E \left(    H \left( \frac{x+S_k}{\sigma\sqrt{n}}\right); \tau_x > k, \nu_n = k  \right) \notag \\ 
& \geq \left(  1 - c n^{\ee} r_{n,\delta}  \right) J_4, 
\end{align}
where $r_{n,\delta}=n^{-\frac{\delta}{2(3+\delta)}}$, $N_n=[\gamma n^{1-2\ee}\log n]$
and $J_4$ is given by  \eqref{term-J5-001}. 
Combining \eqref{start bndtau-001}, \eqref{low-bound J_2-001} and \eqref{low bound J_4-001},
and choosing $\ee=\frac{\delta}{4(3+\delta)}$, 
 we obtain the following lower bound: 
\begin{align} \label{Final lower bnd for Ptau_x-002}
\mathbb{P} \left( \tau_x > n \right) - \frac{V(x) L\left(\frac{x}{\sqrt{n}} \right) }{\sigma\sqrt{n}} 
 \geq -\frac{c}{ n^{\frac{1}{2} +  \frac{\delta}{4} }} 
 -  c\frac{V(x) L\left(\frac{x}{\sqrt{n}} \right) }{n^{\frac{1}{2} + \frac{\delta}{4(3+\delta)}} }  \log n.    
\end{align}
Putting together \eqref{Final upper bnd for Ptau_x-002} and \eqref{Final lower bnd for Ptau_x-002} 
concludes the proof of Theorem \ref{Theor-probtauUN-001}.
\end{proof}

\section{Proof of the conditioned central limit theorem} \label{sec Theor-probIntUN-002}

In this section, we present the proof of Theorem \ref{Theor-probIntUN-002}, 
maintaining the notation and assumptions introduced in Section \ref{SecProof Theor-probtau-001} throughout.

\subsection{The case of large starting points}

The following result establishes a rate of convergence for the joint probability 
$\bb P(\frac{x+S_n }{ \sigma \sqrt{n}} \leq  u, \tau_x >n),$
which becomes particularly effective for large starting point $x$ of the order $\sqrt{n}$. 
Recall that $N_n =[\gamma n^{1-2\ee} \log n]$, 
where $\ee \in [0, \frac{1}{2}]$ and $\gamma>0$ is chosen to be a sufficiently large constant.

\begin{lemma} \label{Lemma-KMTrate-200} \ \\
1. There exists a constant $c >0$ such that, for any $n\geq 1$, $1\leq k \leq N_n$ and $x,u\geq 0$, 
\begin{align*}
\left| \mathbb{P} \left( \frac{x+S_{n-k}}{\sigma\sqrt{n}}  \leq u,  \tau_x > n-k \right) 
- \int_{0}^{u}  \psi \left(\frac{x}{\sigma\sqrt{n}} , y   \right)  dy  \right|  
\leq   c \left(n^{-\frac{\delta}{2(3+\delta)}} + \frac{N_n}{n}\right).
\end{align*}
\\
2. Moreover, there exists a constant $c >0$ such that, 
for any $n\geq 1$, $1\leq k \leq N_n$, $u\geq 0$, $a \in (0,1]$  and $x \geq \sigma a \sqrt{n}$, 
\begin{align*}
\left| \mathbb{P} \left( \frac{x+S_{n-k}}{\sigma\sqrt{n}}  \leq u,  \tau_x > n-k \right) 
- \int_{0}^{u}  \psi \left(\frac{x}{\sigma\sqrt{n}} , y   \right)  dy  \right|  
\leq \frac{c}{a} H \left( \frac{x}{\sigma\sqrt{n}} \right) \left(n^{-\frac{\delta}{2(3+\delta)}} + \frac{N_n}{n}\right).
\end{align*}

\end{lemma}
\begin{proof}
Let $s>0$ and 
denote $A_{n}=\left \{ \sup_{0\leq t\leq 1} \left| S_{[nt]} - \sigma B_{nt} \right| \leq \sigma n^{1/2 - s} \right\}$.
By  Lemma \ref{FCLT}, for any $s>0$ and $n\geq 1$, we have $\bb P(A_n^c) \leq c n^{-\frac{\delta}{2} + s(2+\delta)}$.
Let $m=n-k$ and $S_m(t) = S_{[mt]}$. 
Then, using Lemma \ref{lemma tauBM}, we get that, for any $x\geq - \sigma n^{\frac{1}{2}-s}$, 
\begin{align}  \label{Levyden001}
&\mathbb{P} \left( \frac{x+S_{n-k}}{\sigma\sqrt{n}}  \leq u,  \tau_x > n-k \right) \notag \\
&\leq  \mathbb{P} \left( x+S_{m}(1) \leq u \sigma\sqrt{n},  \sup_{0\leq t\leq 1}\left(x + S_m(t)\right) \geq 0, A_n  \right) + \bb P(A_n^c) \notag\\
&\leq \mathbb{P} \left( x - \sigma n^{\frac{1}{2}-s} + \sigma B_{m} \leq u \sigma\sqrt{n},  
\sup_{0\leq t\leq 1}\left(x+\sigma n^{\frac{1}{2}-s} + \sigma B_{mt}\right) \geq 0\right) 
+ \bb P(A_n^c)\notag \\
& = \bb P \left( x + \sigma n^{\frac{1}{2}-s} +  \sigma B_{m} \in \left[0,  u \sigma\sqrt{n} + 2 \sigma n^{\frac{1}{2} -s} \right],  
\tau_{x + \sigma n^{\frac{1}{2}-s}}^{\sigma B}> m \right)  + \bb P(A_n^c)  \notag\\
&\leq   \int_{0}^{(u \sqrt{n} +  2n^{\frac{1}{2}-s})/\sqrt{m} } \psi\left(\frac{x+\sigma n^{\frac{1}{2}-s}}{\sigma \sqrt{m}},y\right) dy
+ \frac{c}{n^{\frac{\delta}{2} - s(2+\delta) }}, 
\end{align}
where the function $\psi$ is defined by \eqref{Def-Levydens}. 
Since $\psi(x,y) = \phi(y-x) -\phi(x+y)$,  we have
\begin{align} \label{Levyden002}
I &: = \int_{0}^{(u \sqrt{n} + 2 n^{\frac{1}{2}-s})/\sqrt{m} } \psi\left(\frac{x+\sigma n^{\frac{1}{2}-s}}{\sigma\sqrt{m}},y\right) dy 
- \int_{0}^{u } \psi\left(\frac{x}{\sigma\sqrt{n}},y\right) dy \notag \\
&= \int_{0}^{(u \sqrt{n} + 2 n^{\frac{1}{2}-s})/\sqrt{m} } \phi\left( y - \frac{x+\sigma n^{\frac{1}{2}-s} }{\sigma\sqrt{m}} \right) dy 
-  \int_{0}^{u } \phi\left( y -\frac{x}{\sigma \sqrt{n}} \right) dy \notag \\
& \quad - \left(\int_{0}^{(u \sqrt{n} + 2 n^{\frac{1}{2}-s})/\sqrt{m} } \phi\left( y + \frac{x+\sigma n^{\frac{1}{2}-s}}{\sigma\sqrt{m}} \right) dy
-  \int_{0}^{u } \phi\left( y + \frac{x}{\sigma \sqrt{n}} \right) dy  \right). 
\end{align}
By a change of variable, we obtain
\begin{align} \label{Levyden002bb}
I &= \int_{- \frac{x+\sigma n^{\frac{1}{2}-s} }{\sigma\sqrt{m}}}^{\frac{u \sigma \sqrt{n} -x + \sigma n^{\frac{1}{2}-s}}{\sigma\sqrt{m}} } 
\phi\left( y \right) dy 
-  \int_{-\frac{x}{\sigma \sqrt{n}}}^{u -\frac{x}{\sigma \sqrt{n}} } \phi(y) dy \notag \\
& \quad - \left(\int_{\frac{x+\sigma n^{\frac{1}{2}-s} }{\sigma\sqrt{m}}}^{\frac{u \sigma \sqrt{n} + x + 3\sigma n^{\frac{1}{2}-s}}{\sigma\sqrt{m}} } 
\phi\left( y \right) dy
-  \int_{\frac{x}{\sigma \sqrt{n}} }^{u + \frac{x}{\sigma \sqrt{n}}  } \phi\left( y \right) dy  \right) \notag \\
&=: I_1 - I_2. 
\end{align}
For $I_1$, using the fact that $\frac{m}{n}=1+O(N_n/n)$ and Lemma \ref{lem-inequality for H}, we have
\begin{align*} 
| I_1 |  & \leq  \frac{1}{2}\left| H\left(\frac{u \sigma \sqrt{n} -x + \sigma n^{\frac{1}{2}-s}}{\sigma\sqrt{m}} \right) 
         - H\left(\frac{u \sigma \sqrt{n} -x}{\sigma\sqrt{n}} \right) \right| \\
& \quad + \frac{1}{2}\left| H\left(-\frac{ x + \sigma n^{\frac{1}{2}-s}}{\sigma\sqrt{m}} \right) 
         - H\left(-\frac{x}{\sigma\sqrt{n}} \right) \right|
\leq c \left(n^{-s} + \frac{N_n}{n}\right).
\end{align*}
similarly, 
for $I_2$, one can also show that $|I_2| \leq c \left(n^{-s} + \frac{N_n}{n}\right)$. 
Taking $s= \frac{\delta}{2(3+\delta)}$ and combining these bounds 
with \eqref{Levyden001} and \eqref{Levyden002},  the first assertion of the lemma follows.

The second assertion is a direct consequence of the first one, since $H ( \frac{x}{\sigma\sqrt{n}} ) \geq H(a) \geq c a$ for some $c>0$,  
whenever $a\in (0,1]$ and $x \geq \sigma a n^{\frac{1}{2}}$. 
\end{proof}

\subsection{Proof of Theorem \ref{Theor-probIntUN-002}}

Let $x\in \bb R$, $u\in \bb R_+$ and $n\geq 1$.
Using the Markov property, we have 
\begin{align}\label{start bndtau-201}
\mathbb{P} \left( \frac{x+S_n }{ \sigma \sqrt{n}} \leq  u,  \tau_x > n \right)  = I_1 + I_2, 
\end{align}
where 
\begin{align*}
 I_1 & = \mathbb{P} \left( \frac{x+S_n }{ \sigma \sqrt{n}} \leq  u,  \tau_x > n, \nu_{n} > N_n \right),  \notag\\
 I_2 &= \mathbb{P} \left( \frac{x+S_n }{ \sigma \sqrt{n}} \leq  u,  \tau_x > n, \nu_{n} \leq N_n \right)  \notag\\
& = \sum_{k = 1}^{N_n}
  \int_{M_{n}}^\infty \mathbb{P}\left( \frac{x'+S_{n-k} }{ \sigma \sqrt{n}} \leq  u, \tau_{x'}>n-k \right)  
\mathbb{P} \left( x + S_k\in dx', \tau_x > k, \nu_{n} = k \right), 
\end{align*}
where $M_n = n^{1/2-\ee}$ and $N_n=[\gamma n^{1-2\ee} \log n]$, with $\ee \in (0, \frac{1}{2})$
and $\gamma>0$ a sufficiently large constant. 

\textit{Bound of $I_1$}. 
By Lemma \ref{Lemma 2-0}, for any $\gamma>0$, 
there exists a constant $c>0$ such that for any $n\geq 1$ and $x \in \bb R$, 
\begin{align} \label{upper bound-I1-201}
I_1 \leq  \mathbb{P} \left( \nu_{n} > N_n \right)  \leq  c n^{-\gamma/2}.  
\end{align}

\textit{Upper bound of $I_{2}$}. 
Let $1 \leq k \leq N_n=[\gamma n^{1-2\ee} \log n]$ and $x' \geq M_n = n^{1/2-\ee}$. 
Using Lemma \ref{Lemma-KMTrate-200} with $a=n^{-\ee}$,  we get 
\begin{align} 
\label{Probab-Int tau-201}
  \mathbb{P}\left(  \frac{x'+S_{n-k} }{ \sigma \sqrt{n}} \leq  u,  \tau_{x'} > n-k\right)  
&   \leq   \int_{0}^{u}  \psi \left( \frac{x'}{\sigma \sqrt{n}}, y   \right)  dy  
+ c n^{\ee} H\left(\frac{x'}{\sigma \sqrt{n}}\right) \left( r_{n,\delta} + \frac{N_n}{n}\right),
\end{align}
where $r_{n,\delta}= n^{-\frac{\delta}{2(3+\delta)}}$ is the remainder term in Lemma \ref{Lemma-KMTrate-200}.
Substituting \eqref{Probab-Int tau-201} into $I_2$ gives  
\begin{align} \label{bound I_2-002}
I_{2}  &\leq  
 \sum_{k = 1}^{N_n} \bb E \Bigg(    \int_{0}^{u}  \psi \left( \frac{x+S_k}{\sigma \sqrt{n}}, y   \right)  dy
 + c n^{\ee} H\left(\frac{x+S_k}{\sigma \sqrt{n}}\right)  \left(r_{n,\delta} + \frac{N_n}{n} \right);    \tau_x > k, \nu_n = k  \Bigg) \notag \\ 
&=  \sum_{k = 1}^{N_n} \bb E \Bigg(    \int_{0}^{u}  \psi \left( \frac{x+S_k}{\sigma \sqrt{n}}, y   \right)  dy
 + c n^{\ee} H\left(\frac{x+S_k}{\sigma \sqrt{n}}\right)  \left(r_{n,\delta} + \frac{N_n}{n} \right); \notag \\
&\qquad\qquad\qquad\qquad\qquad\qquad\qquad\qquad\qquad\qquad \frac{|S_k|}{\sigma\sqrt{n}}  > \alpha n^{-\ee}\log n,   \tau_x > k, \nu_n = k  \Bigg) \notag \\ 
 &\quad  +  
 \sum_{k = 1}^{N_n} \bb E \Bigg(    \int_{0}^{u}  \psi \left( \frac{x+S_k}{\sigma \sqrt{n}}, y   \right)  dy
 + c n^{\ee} H\left(\frac{x+S_k}{\sigma \sqrt{n}}\right) \left(r_{n,\delta} + \frac{N_n}{n} \right); \notag \\
&\qquad\qquad\qquad\qquad\qquad\qquad\qquad\qquad\qquad\qquad \frac{|S_k|}{\sigma\sqrt{n}}  \leq  \alpha n^{-\ee}\log n, \tau_x > k, \nu_n = k  \Bigg) \notag \\
 & =: I_3 + I_4. 
\end{align}
By \eqref{H as int od Dirichlet kern-001}, we have that, on the event $\{\tau_x > k\}$, 
\begin{align*}
\int_{0}^{u}  \psi \left( \frac{x+S_k}{\sigma \sqrt{n}}, y   \right)  dy 
\leq \int_{0}^{\infty}  \psi \left( \frac{x+S_k}{\sigma \sqrt{n}}, y   \right)  dy 
=  H\left(\frac{x+S_k}{\sigma \sqrt{n}}\right). 
\end{align*}
Hence, applying Lemma \ref{Bound for J_4}, we get that for any $\ee \in (0, \frac{\delta}{2(3 + \delta)})$, 
\begin{align} \label{bound I_3-001}
I_3
&\leq \left(1+cn^{\ee}r_{n,\delta} + cn^{\ee}\frac{N_n}{n} \right) \sum_{k = 1}^{N_n} \bb E  \left( H  \left( \frac{x+S_k}{\sigma \sqrt{n}} \right); 
 \frac{|S_k|}{\sigma\sqrt{n}}  > \alpha n^{-\ee}\log n,   \tau_x > k, \nu_n = k  \right) \notag\\
& = \left(1+cn^{\ee}r_{n,\delta} + cn^{\ee}\frac{N_n}{n} \right) J_3 \notag\\ 
&\leq c  \frac{1 + V(x) L\left(\frac{x}{\sigma\sqrt{n}}\right)}{ n^{\frac{1}{2} + \frac{\delta}{2}  - \delta\ee  } } (\log n)^{3/2}.
\end{align}
Now we deal with $I_4$. We have
\begin{align*} 
I_4 & =  \sum_{k = 1}^{N_n} \bb E \left(    \int_{0}^{u}  \psi \left( \frac{x+S_k}{\sigma \sqrt{n}}, y   \right)  dy; 
 \frac{|S_k|}{\sigma\sqrt{n}}  \leq  \alpha n^{-\ee}\log n, \tau_x > k, \nu_n = k  \right) \notag \\
& \quad + c n^{\ee} \left(r_{n,\delta} + \frac{N_n}{n} \right)\sum_{k = 1}^{N_n} \bb E \left( H\left(\frac{x+S_k}{\sigma \sqrt{n}}\right) ; 
 \frac{|S_k|}{\sigma\sqrt{n}}  \leq  \alpha n^{-\ee}\log n, \tau_x > k, \nu_n = k  \right) \notag \\
&=: I_{41} + I_{42}.
\end{align*}
By \eqref{def of func ell_h-001}, we have $\psi(x',y)=x'\ell_h(x',y)$ for any $x',y\in \bb R$, so that 
\begin{align*} 
I_{41}  =  \sum_{k = 1}^{N_n}  
\bb E \Bigg(\frac{x+S_k }{\sigma\sqrt{n}} \int_0^{u}\ell_h\left(\frac{x+S_k  }{\sigma\sqrt{n}},y \right) dy; 
 \frac{|S_k|}{\sigma \sqrt{n}} \leq \alpha n^{-\ee}\log n, \tau_x > k, \nu_n = k\Bigg).
\end{align*}
On the set $\left\{ \frac{|S_k|}{\sigma \sqrt{n}} \leq \alpha n^{-\ee}\log n \right\}$, 
by Lemma \ref{Holder prop for int ell}, it holds that, for $n$ sufficiently large,  
\begin{align*} 
&\left| \int_0^{u}\ell_h\left(\frac{x+S_k  }{\sigma\sqrt{n}},y \right) dy - \int_0^{u}\ell_h\left(\frac{x}{\sigma\sqrt{n}},y \right) dy \right| 
 \leq c L\left(\frac{x}{\sigma \sqrt{n}}\right) n^{-\ee}\log n.
\end{align*}
This implies
\begin{align} \label{bound I_41-001}
 I_{41} \leq  
& \frac{1 }{\sigma\sqrt{n}}  \left(  \int_0^{u} \ell_h\left(\frac{x }{\sigma\sqrt{n}},y \right) dy 
  +  c  L\left(\frac{x}{\sigma \sqrt{n}}\right) n^{-\ee}\log n \right)  \notag\\
& \quad \times  \sum_{k = 1}^{N_n}  
 \bb E \Bigg(x+S_k  ; 
    \frac{|S_k|}{\sigma \sqrt{n}} \leq \alpha n^{-\ee}\log n, \tau_x > k, \nu_n = k\Bigg) \notag\\
& \leq  \frac{1 }{\sigma\sqrt{n}} \left(  \int_0^{u} \ell_h\left(\frac{x }{\sigma\sqrt{n}},y \right) dy 
  +  c  L\left(\frac{x}{\sigma \sqrt{n}}\right) n^{-\ee}\log n \right) E_{x, n, \ee}  \notag\\
& \leq  \frac{V(x) }{\sigma\sqrt{n}} \left(\int_0^{u}\ell_h\left(\frac{x }{\sigma\sqrt{n}},y \right) dy 
+ c L\left(\frac{x}{\sigma \sqrt{n}}\right)  n^{-\ee}\log n\right), 
\end{align}
where in the last inequality we used Lemma \ref{lemma E1 and E2}.

For $I_{42}$, we use the bound \eqref{final bound for J4-001bb} of Lemma \ref{Lem-bound-UL-J4} to get 
\begin{align} \label{bound I_42-001}
I_{42} \leq c n^{\ee} \left(r_{n,\delta} + \frac{N_n}{n} \right) \frac{V(x) L\left(\frac{x}{\sigma\sqrt{n}} \right) }{\sigma\sqrt{n}}. 
\end{align}
Putting together \eqref{start bndtau-201} --  \eqref{bound I_42-001}, 
letting $\ee=\frac{\delta}{4(3+\delta)} < \frac{1}{4}$ and choosing  $\gamma>0$ to be sufficiently large, 
there exists a constant $c>0$ such that
\begin{align} \label{final upper bound PPtau-001}
\mathbb{P} \left( \frac{x+S_n }{ \sigma \sqrt{n}} \leq  u,  \tau_x > n \right)  
& \leq  c  \frac{1 + V(x) L\left(\frac{x}{\sigma\sqrt{n}}\right)}{ n^{\frac{1}{2} + \frac{\delta}{2}  - \delta\ee  } } (\log n)^{3/2}   \notag \\
& \quad +  \frac{V(x) }{\sigma\sqrt{n}} \left(\int_0^{u}\ell_h\left(\frac{x }{\sigma\sqrt{n}},y \right) dy + c L\left(\frac{x}{\sigma\sqrt{n}}\right) 
 n^{-\ee}\log n\right)\notag  \\
& \quad + c n^{\ee} \left(r_{n,\delta} + \frac{N_n}{n} \right) \frac{V(x) L\left(\frac{x}{\sigma\sqrt{n}} \right) }{\sigma\sqrt{n}} \notag \\
&\leq \frac{V(x) }{\sigma\sqrt{n}} \int_0^{u}\ell_h\left(\frac{x }{\sigma\sqrt{n}},y \right) dy 
+ \frac{ c}{ n^{\frac{1}{2} +  \frac{\delta}{4} }} \notag \\
& \quad + c  \frac{V(x) L\left(\frac{x}{\sigma \sqrt{n}}\right) }{\sigma\sqrt{n}}  n^{-\frac{\delta}{4(3+\delta)}}\log n,
\end{align}
which proves the upper bound, when $n$ is large enough. 
For $n$ small (i.e.\ $n\leq n_0$, where $n_0>0$ is a positive constant) the bound  \eqref{final upper bound PPtau-001}
is trivial, by choosing the constant $c$ large enough.  
Proceeding in the same way as for the term $I_4$, we prove the lower bound
\begin{align} \label{final lower bound PPtau-001}
\mathbb{P} \left( \frac{x+S_n }{ \sigma \sqrt{n}} \leq  u,  \tau_x > n \right)  
\geq & \ \frac{V(x) }{\sigma\sqrt{n}} \int_0^{u}\ell_h\left(\frac{x }{\sigma\sqrt{n}},y \right) dy 
- \frac{ c}{ n^{\frac{1}{2} +  \frac{\delta}{4} }} \notag \\
& - c  \frac{V(x) L\left(\frac{x}{\sigma\sqrt{n}}\right) }{\sigma\sqrt{n}}  n^{-\frac{\delta}{4(3+\delta)}}\log n.
\end{align}
Combining \eqref{final upper bound PPtau-001}, \eqref{final lower bound PPtau-001} and using \eqref{def of func h-001}
finishes the proof of Theorem \ref{Theor-probIntUN-002}.


\vskip5mm
\end{document}